\documentclass[12pt]{article}
\usepackage{amssymb, amsmath, hyperref}
\usepackage{amsthm}
\usepackage{bm}
\usepackage[margin=0.9in]{geometry}
\usepackage{enumerate}
\allowdisplaybreaks[1]
\numberwithin{equation}{section}

\newtheorem{notation}{Notation}[section]
\newtheorem{example}{Example}[section]
\newtheorem{thm}{Theorem}[section]
\newtheorem{cor}{Corollary}[section]
\newtheorem{note}{Note}[section]

\newtheorem{defn}{Definition}[section]

\newtheorem{lemma}{Lemma}[section]
\usepackage{graphicx}

\begin{document}
\vspace{-2cm}
\markboth{R.Vishnupriya and R. Rajkumar}{New matrices for spectral hypergraph theory, I}
\title{\LARGE\bf 
New matrices for spectral hypergraph theory, I}
\author{R. Vishnupriya\footnote{e-mail: {\tt rrvmaths@gmail.com},} \footnote{First author is supported by University Grands Commission (U.G.C.), Government of India under the Fellowship grant No. 221610053976,},~
R. Rajkumar\footnote{e-mail: {\tt rrajmaths@yahoo.co.in}.} \\
{\footnotesize Department of Mathematics, The Gandhigram Rural Institute (Deemed to be University),}\\ \footnotesize{Gandhigram -- 624 302, Tamil Nadu, India}\\[3mm]
}
\date{}
\maketitle
\begin{abstract}
	We introduce a hypergraph matrix, named the unified matrix, and use it to represent the hypergraph as a graph.  We show that the unified matrix of a hypergraph is identical to the adjacency matrix of the associated graph. This enables us to use the spectrum of the unified matrix of a hypergraph as a tool to connect the structural properties of the hypergraph with those of the associated graph.
Additionally, we introduce certain hypergraph structures and invariants during this process, and relate them to the eigenvalues of the unified matrix.
	
\end{abstract}

\section{Introduction} \label{sec1}
Spectral graph theory studies the properties of graphs through the eigenvalues and eigenvectors of matrices associated with the graph, such as adjacency matrix, Laplacian matrix, signless Laplacian matrix and normalized Laplacian matrix (see~\cite{bapat,chung,cvetko}). A hypergraph is a generalization of a graph where edges (also called hyperedges) can connect more than two vertices.
In the study of spectral hypergraph theory, tensors, in addition to matrices, have also been associated with hypergraphs, over the past decades.
The association of tensors with hypergraphs was made by  Bul\`{o} and Pelillo~\cite{Bulo}, Cooper and Dutle~\cite{cooper2012spectra}, Hu and Qi~\cite{Hu}, Li et al.~\cite{Li}, Xie and Chang~\cite{Xie},  Qi~\cite{Qi} and Anirban Banerjee et al.~\cite{Banarjee tensor}. For further information on the study of the spectra of tensors associated with hypergraphs, we direct the reader to~\cite{Qi tensor bk}. However,  the computational complexity of calculating eigenvalues is an NP-hard problem. Additionally, not all aspects of spectral graph theory can be seamlessly extended to spectral hypergraph theory.  These limitations are challenges when studying spectral hypergraph theory using tensors. Association of matrices with hypergraphs was made by Feng and Li~\cite{ad mat 1}, Rodr\'{\i}guez~\cite{Rod 2002},  Reff and Rusnak~\cite{reff2012}, Anirban Banerjee~\cite{banerjee2021spectrum}, and Cardoso  and Trevisan~\cite{cardoso}. Analyzing the structural properties of a hypergraph through the spectra of its associated matrix addresses the limitations inherent in tensor-based methods. In addition, the representation of a hypergraph as a weighted graph using the associated matrix has sometimes been used to derive spectral properties of the hypergraph. It is known that a graph is uniquely determined by its adjacency matrix, and similarly, a hypergraph can be uniquely determined by the associated tensors mentioned above. However, this may not hold true for the associated matrices mentioned above.

To address this, in this paper, we introduce the unified matrix for a hypergraph, which is identical to the adjacency matrix of the associated graph. This allows us to use the unified matrix's spectrum to link the structural properties of the hypergraph with those of the graph. We also introduce certain hypergraph structures and invariants, relating them to the eigenvalues of the unified matrix.

The rest of the paper is organized as follows. Section~\ref{sec2} presents some basic definitions and notations related to graphs, hypergraphs and matrices.
 In Section~\ref{sec3}, we introduce the unified matrix of a hypergraph, and define the associated graph of this matrix. Section~\ref{sec4} provides some basic results on the unified matrix and unified eigenvalues of a hypergraph. In Section~\ref{sec5}, we define new structures and invariants on hypergraphs, such as exact walk, exact path, exact cycle, unified path, unified cycle, exactly connectedness, exact girth, exact distance, and exact diameter. We also provide bounds on the exact diameter of a hypergraph using the eigenvalues of its unified matrix. Further we characterize simple hypergraphs having no odd exact cycle using the characteristic polynomial and the spectrum of its unified matrix. Section~\ref{sec6} establishes a formula for the determinant of the unified matrix, and describes some relationships between the characteristic polynomial of the unified matrix of a hypergraph and its structural properties. In Section~\ref{sec7}, we determine the spectrum of unified cycles and unified paths. Section~\ref{sec8} presents the characteristic polynomial of the unified matrix of a hypergraph constructed by some elementary hypergraph operations. In Section~\ref{sec9}, we bound the unified spectral radius of a hypergraph by some hypergraph invariants. We also bound the strong chromatic number and weak chromatic number of a hypergraph by the largest and smallest eigenvalues of its unified matrix. Finally, we bound the independence number and the complete clique number of a hypergraph using its number of non-negative, positive, and negative unified eigenvalues.


\section{Preliminaries}\label{sec2}
Let $S$ be a non-empty set. We denote the set of all non-empty subsets of $S$ by $\mathcal{P}^*(S)$.
A \textit{hypergraph} $H$ consists of a non-empty set $V(H)$ and a multiset $E(H)$ of non-empty subsets of $V(H)$.
The elements of $V(H)$; are called \textit{vertices} and the elements of $E(H)$ are called \textit{hyperedges}, or simply \textit{edges} of $H$. $H$ is said to be \textit{uniform} if all of it’s edges have the same cardinality; If it is $m$, then $H$ is said
to be \textit{$m$-uniform}. The \textit{rank} of $H$, denoted by $rank(H)$, is defined as $\displaystyle \underset{e\in E(H)}{\max}\{|e|\}$, if $E(H)\neq\Phi$; 0, otherwise.
 An edge $e$ of $H$ is called \textit{included} if there exist an edge $e'(\neq e)\in E(H)$ such that $e\subset e'$.
  The \textit{multiplicity} of an edge $e$ in $H$, denoted by $m(e)$, is the number of occurrences of that edge in $H$. An edge $e$ of $H$ is called \textit{multiple} if $m(e)\geq2$.  An edge of $H$ having cardinality one is called a \textit{loop}. The \textit{degree} of a vertex $v$ in $H$, denoted by $d_H(v)$, is the number of edges in $H$ containing $v$.
$\delta(H)$ and $\Delta(H)$ denote the minimum and the maximum of the degrees of vertices in $H$, respectively. 

$H$ is said to be \textit{simple} if it has neither loops nor multiple edges. $H$ is said to be \textit{complete} if $E(H)=\mathcal{P}^*(V(H))$.  A \textit{subhypergraph} of $H$ is a hypergraph $H'$ with $V(H')\subseteq V(H)$ and $E(H')\subseteq E(H)$. A subhypergraph $H'$ of $H$ is called \textit{induced} if all the edges of $H$ which are completely contained in $V(H')$ forms the edge set $E(H')$.  
For each $T\subseteq V(H)$, $H\backslash T$ denotes the subhypergrph of $H$ obtained from $H$ by deleting all the vertices in $T$ and all the edges $e\in E(H)$ such that $T\subseteq e$.
Let $H_1,H_2,\dots,H_r$ be mutually disjoint hypergraphs. Then the \emph{disjoint union of $H_is$}, denoted by $\dot{\bigcup}H_i$, is the hypergraph having the vertex set $\underset{i=1}{\overset{r}{\bigcup}}V(H_i)$ and the edge set $\underset{i=1}{\overset{r}{\bigcup}}E(H_i)$. 

A hypergraph $H$ is said to be a \textit{path} if its edges $e_1,e_2,\dots, e_{t}$ and some of its distinct vertices  $v_0,v_1,v_2,\dots,v_{t-1}$ can be arranged in a sequence 
$v_0, e_1,v_1, e_2, v_2,\dots,v_{t-1}, e_{t}, v_t$
in such a way that $v_{i-1},v_{i} \in e_i$ for all $i=1,2,\ldots,t$. $H$ is said to be a \textit{cycle} if its edges $e_1,e_2,\dots, e_{t}$ and some of its distinct vertices  $v_0,v_1,v_2,\dots,v_{t-1}$ can be arranged in a sequence 
$v_0, e_1,v_1, e_2, v_2,\dots,v_{t-1}, e_{t}, v_0$
in such a way that $v_{i-1},v_{i} \in e_i,$ for all $i=1,2,\ldots,t-1$
and $v_{t-1},v_{0} \in e_1$.

Now, we state some definitions and notations on graphs:  The \textit{distance} between the vertices $u$ and $v$ in a graph $G$, denoted by $d_G(u,v)$, is the length of a shortest path joining $u$ and $v$. If there is no path joining $u$ and $v$ in $G$, then $d_G(u,v)$ is defined to be $\infty$. The \textit{diameter} of a connected graph $G$ is the greatest distance between any pair of vertices in $G$.
A subgraph $G'$ of a graph $G$ is called an \textit{elementary subgraph} if each component of $G'$ is either an edge or a cycle. 
$P_n$, $C_n$ and $K_n$ denote the path, the cycle and the complete graph on $n$ vertices, respectively. The complete bipartite graph with bipartitions of cardinality $m$ and $n$ is denoted by $K_{m,n}$.

Let $G$ be a finite, loopless graph with $V(G)=\{v_1,v_2,\dots,v_n\}$.  
 The \textit{adjacency matrix} of $G$, denoted by $A(G)$, is the matrix of order $n$ whose rows and columns are indexed by the vertices of $G$ and for all $v_i, v_j\in V(G)$,
	\begin{center}
		the $(v_i, v_j)^{th}$ entry of $A(G)=$
		$\begin{cases}
		m(\{v_i,v_j\}), &\text{if}~i\neq j~\text{and}~\{v_i,v_j\}\in E(G);\\
		0,&\text{otherwise}.
		\end{cases}$ 
	\end{center}
 
We denote the identity matrix of order $n$ by $I_n$, and the all-ones matrix of size $n\times m$ by $J_{n\times m}$.  We denote $\mathbf{0}$ as the zero matrix of appropriate size. 
Let $M$ be an $m\times n$ matrix. For $S\subseteq \{1,2\dots,m\}$ and $T \subseteq \{1,2\dots,n\}$, $M[S|T]$ denotes the submatrix of $M$ determined by the rows corresponding to $S$ and the columns corresponding to $T$. If $S$ and $T$ are proper subsets, then the submatrix obtained by deleting the rows in $S$ and the columns in $T$ is denoted by $M(S|T)$. If $S=\{i\}$ and $T=\{j\}$, then we denote $M[S|T]$ and $M(S|T)$ simply by $M[i|j]$  and $M(i|j)$, respectively. $r(M)$ denotes the rank of $M$. If $M$ is a square matrix of order $n$, then the trace of $M$ is denoted by $tr(M)$, and the characteristic polynomial of $M$ is denoted by $P_M(x)$. The eigenvalues of $A$ are denoted by $\lambda_i(M)$, $i=1,2,\dots,n$. The spectrum of $M$ is the multiset of eigenvalues of $M$ and is denoted by $\sigma(M)$. 

In the rest of the paper we consider only hypergraphs having finite number of vertices.


\section{Unified matrix of a hypergraph}\label{sec3}
In this section we introduce the unified matrix of a hypergraph and we use it to associate a graph to the hypergraph.

Let $S$ be a non-empty set. If $\{S_1,S_2\}$ is a $2$-partition of $S$, then we call $S_1$ and $S_2$ as \textit{parts} of $S$. Let $\tau(S)$ denote the set of all $2$-partitions of $S$. 
\begin{defn}\normalfont
	Let $H$ be a hypergraph. Let $I(H)$ denote the set of all the parts of each edge of $H$ together with all the singleton subsets of $V(H)$. We call the cardinality of $I(H)$ as the \textit{edge index} or simply the \textit{$e$-index} of $H$. 
\end{defn}
\begin{defn}\normalfont
	Let $H$ be a hypergraph. Let $S$ and $S'\in I(H)$. We say that \textit{$S$ is a neighbor of $S'$ with multiplicity} $c~(\geq1)$, if there is an edge $e$ in $H$ with multiplicity $c$ for which $S$ and $S'$ are parts. We denote it by $S\overset{c}{\sim}S'$. Also, we simply write $S\sim S'$, whenever there is no necessity to mention the multiplicity of $S$ explicitly. 
\end{defn}

Now, we define the unified matrix of a hypergraph as follows.
\begin{defn}\normalfont
	Let $H$ be a hypergraph with $e$-index $k$. Let $I(H)=\{S_1, S_2,\dots, S_{k}\}$. Then \textit{the unified matrix of $H$}, denoted by $\mathbf{U}(H)$, is the matrix of order $k$ whose rows and columns are indexed by the elements of $I(H)$ and for all $S_i, S_j\in I(H)$, 
	\begin{center}
	 the $(S_i, S_j)^{th}$ entry of $\mathbf{U}(H)=$
		$\begin{cases}
		m(\{S_i\}), &\text{if}~i=j,~|S_i|=1~\text{and}~S_i\in E(H);\\
		c, &\text{if}~S_i \overset{c}{\sim} S_j;\\
		0,&\text{otherwise}.
		\end{cases}$
		\end{center}
\end{defn}
\begin{example}\normalfont
Consider the hypergraph $H$ shown in Figure~\ref{fig}. 
	\begin{figure}[ht]
	\begin{center}
		\includegraphics[scale=1]{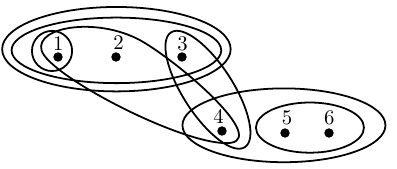}
	\end{center}\caption{The hypergraph $H$}\label{fig}
\end{figure}
Notice that $I(H)=\{\{1\}, \{2\}, \{3\},$ $ \{4\}, \{5\}, \{6\}, \{1,2\}, \{1,3\}, \{1,4\}, \{2,3\},\{2,4\}, \{4,5\},$ $\{4,6\}, \{5,6\}\}$. We name the elements of $I(H)$ as $S_1, S_2, \dots, S_{14}$ in the order in which they appear in it. Then for $i,j\in \{1,2,\dots,14\}$,
	\begin{center}
	the $(S_i, S_j)^{th}$ entry of $\mathbf{U}(H)=$
	$\begin{cases}
	1, &\text{if}~(i,j)\in A_1;\\
	2, &\text{if}~(i,j)\in A_2;\\
	0,&\text{otherwise},
	\end{cases}$
\end{center} where $A_1=\{(1,1),(1,11),(11,1),(2,9),(9,2),(3,4),(4,3),(4,7),(7,4),(4,14),(5,6),(6,5),(5,13),$\\$(6,12),(12,6)\}$ and $A_2=\{(1,10),(10,1),(2,8),(8,2),(3,7),(7,3)\}$.
\end{example}
 Since $\mathbf{U}(H)$ is a symmetric matrix, its eigenvalues are all real. We denote them by $\lambda_1(H)$, $\lambda_2(H)$,$\dots$, $\lambda_k(H)$ and we shall assume that $\lambda_1(H)\geq\lambda_2(H)\geq\dots\geq\lambda_{k}(H)$. The characteristic polynomial of $\mathbf{U}(H)$ is said to be the \emph{unified characteristic polynomial of $H$}. An eigenvalue of $\mathbf{U}(H)$ is said to be an \emph{unified eigenvalue of $H$}, and
the spectrum of $\mathbf{U}(H)$ is said to be the \emph{unified spectrum of $H$}, or simply \emph{$\mathbf{U}$-spectrum of $H$}.  

Let $H$ be a loopless hypergraph. Since $\mathbf{U}(H)$ is a symmetric matrix of order $k$, it gives raise to a unique loopless graph $G_H$ whose vertex set is $I(H)$ and two vertices $S$ and $S'$ are joined by $c~(\geq1)$ number of edges in $G_H$ if and only if $S \overset{c}{\sim} S'$ in $H$. We call $G_H$ as the \textit{associated graph} of $H$. It is clear that $A(G_H)=\mathbf{U}(H)$, and so they have the same spectrum. 
Moreover, if $H$ is a loopless graph, then $A(H)$ and $\mathbf{U}(H)$ have the same eigenvalues. Thereby, we denote the eigenvalues of these two matrices commonly as $\lambda_i(H)$, $i=1,2,\dots,k$.
These reveals that the unified matrix of a loopless hypergraph is a natural generalization of the adjacency matrix of a loopless graph. Furthermore, it can be seen that a hypergraph can be uniquely determined by its unified matrix, as long as the indices of the matrix are specified.
\begin{example}\normalfont
	Let $H'$ be the hypergraph obtained by deleting the loop $\{1\}$ in the hypergraph $H$ shown in Figure~\ref{fig}. The associated graph $G_{H'}$ of $H'$ is shown in Figure~\ref{Fig}. 
	\begin{figure}[ht]
		\begin{center}
			\includegraphics[scale=1]{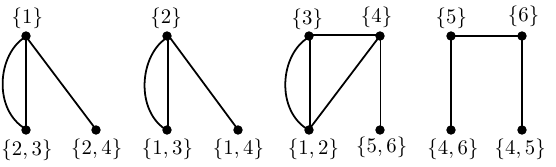}
		\end{center}\caption{The associated graph $G_{H'}$ of $H'$}\label{Fig}
	\end{figure}
\end{example}
From this point forward, the elements of $I(H)$ will only be referenced explicitly when necessary.


 \section{Basic results on the unified matrix of a hypergraph}\label{sec4}
 \begin{defn}\normalfont
	Let $H$ be a hypergraph. For each $S\in I(H)$, we define the \textit{unified degree of $S$ in $H$}, denoted by $d_H(S)$, as the cardinality of the multiset $\{e\in E(H)~|~S\subseteq e\}$. Further, we define the \textit{minimum unified degree of $H$}, denoted by $\delta^*(H)$, as $\delta^*(H)=\min\{d_H(S)\mid S\in I(H)\}$.
\end{defn}
 It is clear that if $S=\{v\}$, then $d_H(\{v\})=d_H(v)$.

\begin{notation}\normalfont
	For a hypergraph $H$, we introduce the following notations.
	\begin{itemize}
		\item Let $\partial(H)$ denote the cardinality of the multiset of all included edges of $H$ with cardinality at least 2.
		\item Let $E^*(H)$ denote the set of all elements in the multiset $E(H)$. 
	\end{itemize}
\end{notation}
\begin{lemma}\normalfont
	Let $H$ be a hypergraph with $e$-index $k$ and let $\mathbf{U}(H)=(\mathbf{U}_{S_iS_j})_{1\leq i,j\leq k}$ be its unified matrix. Then we have the following.
\begin{itemize}
    	\item[(i)]   	
     	The row sum of $\mathbf{U}(H)$ corresponding to a singleton element $\{v\}$ in $I(H)$ gives the degree of the vertex $v$ in $H$. 
	\item[(ii)]  $\underset{v\in V(H)}{\sum}\mathbf{U}_{\{v\}S'}=\underset{v\in V(H)}{\sum}d_H(v)=\underset{e\in E(H)}{\sum}|e|$.  
	
	If $H$ is $m$-uniform, then 
	\begin{align*}
	\underset{v\in V(H)}{\sum}\mathbf{U}_{\{v\}S'}=\underset{v\in V(H)}{\sum}d_H(v)=m\cdot|E(H)|.
	\end{align*}
	\item[(iii)] $ \underset{\lambda\in \sigma(\mathbf{U}(H))}{\sum}\lambda^2=2\left(\underset{e\in E^*(H)}{\sum} m(e)^2|\tau(e)|\right)+\underset{\{v\}\in E^*(H)}{\sum}m(\{v\})^2$; 
	
	In particular, if $H$ has no multiple edges, then 
	\begin{align}\label{L eq2}
\underset{\lambda\in \sigma(\mathbf{U}(H))}{\sum}\lambda^2=\left(\underset{S\in I(H)}{\sum}d_H(S)\right)-\partial(H).
\end{align}
\end{itemize}
\end{lemma}
\begin{proof}
	\begin{itemize}
		\item[(i)] Let $\{v\}\in I(H)$. Notice that each $e\in E(H)$ with multiplicity $m$ with $v\in e$ contributes $m$ to the entry $\mathbf{U}_{\{v\}S}$, where $S=\{v\}$ or $e\backslash\{v\}$ according as $e$ is a loop or not. Hence the result follows.
		\item[(ii)] From part~(i), the first equality is clear. Notice that an edge $e$ contributes $|e|$ in the sum of the degrees of all the vertices in $e$. Thus, we have the second equality. 
		In particular, if $H$ is $m$-uniform, then the result follows by substituting $|e|=m$ in $\underset{e\in E(H)}{\sum}|e|$.
		\item[(iii)] Let $S_i\in I(H)$, where $i\in\{1,2,\dots,k\}$. Since $\mathbf{U}(H)$ is symmetric, we have
		\begin{equation}\label{EQ1}
		\mathbf{U}(H)^2_{S_iS_i}=\sum_{t=1}^{k}\mathbf{U}_{S_iS_t}\mathbf{U}_{S_tS_i}=\sum_{t=1}^{k}\mathbf{U}_{S_iS_t}^2.
		\end{equation}
		Using \eqref{EQ1}, we have
			\begin{align}\label{EQ2}
		\underset{\lambda\in \sigma(\mathbf{U}(H))}{\sum}\lambda^2=tr(\mathbf{U}(H)^2)&=\sum_{i=1}^{k}\mathbf{U}(H)^2_{S_iS_i}\nonumber\\
		&=\sum_{i=1}^{k}\sum_{t=1}^{k}\mathbf{U}_{S_iS_t}^2\nonumber\\
		&=2\sum_{1\leq i<t\leq k}\mathbf{U}_{S_iS_t}^2+\sum_{i=1}^{k}\mathbf{U}_{S_iS_i}^2.
		\end{align}
		Notice that $$\mathbf{U}_{S_iS_t}=\begin{cases}
		m(e),& \text{if}~S_i\neq S_t~\text{and}~\{S_i,S_t\}\in\tau(e)~\text{for some}~e\in E^*(H);\\
		m(\{v\}),&\text{if}~S_i=S_t=\{v\}~\text{and}~\{v\}\in E^*(H);\\
		0,& \text{otherwise}.
		\end{cases}$$	
		The first part of this result follows by substituting $\mathbf{U}_{S_iS_t}$ in~\eqref{EQ2}.
			
			Now, we shall compute 
			\begin{align}\label{dsum}
			\underset{S\in I(H)}{\sum}d_H(S).
			\end{align} Let $S\in I(H)$. We have the following cases to consider.
			
				\textbf{Case a.} Suppose $S\nsubseteq e$ for any $e\in E(H)$. Then $d_H(S)=0$.
			
		\textbf{Case b.} Suppose $S\subset e$. Then the set $\{S, e\backslash S\}\in \tau(e)$ contributes $m(e)$ in~\eqref{dsum}	corresponding to the counting of $d_H(S)$. Since $e\backslash S$ is also a subset of $e~(\neq e\backslash S)$, the same set $\{S, e\backslash S\}\in\tau(e)$ contributes $m(e)$ in~\eqref{dsum} corresponding to the counting of $d_H(e\backslash S)$. Since, each element in $\tau(e)$ is of the form $\{S, e\backslash S\}$, where $S\subset e$ for some $e\in E(H)$, so each element in $\tau(e)$ contributes $2m(e)$ in~\eqref{dsum}.
			
		\textbf{Case c.} Suppose $S\in E(H)$. Then $S$ is either a loop or an included edge. So, $S$ contributes $1$ in~\eqref{dsum} corresponding to the counting of $d_H(S)$.

	By the above cases,~\eqref{dsum} becomes
		
		 \begin{align}\label{EQ3}
			\underset{S\in I(H)}{\sum}d_H(S)=2\left(\underset{e\in E^*(H)}{\sum}m(e)|\tau(e)|\right)+\underset{\{v\}\in E^*(H)}{\sum}m(\{v\})+\partial(H).
		\end{align}
		
		In particular, if $H$ has no multiple edges, then $E^*(H)=E(H)$ and $m(e)=1$ for all $e\in E(H)$. Substituting these in~\eqref{EQ3} and in the first part of this result and then combining them, we get the result as desired.		
			\end{itemize}
\end{proof}
\begin{thm}
	  If $H$ is a loopless hypergraph, then
	\begin{equation*}
	\underset{T}{\sum} \sqrt{|det(\mathbf{U}(H)[T|T])|}=\underset{e\in E^*(H)}{\sum} m(e)|\tau(e)|,
	\end{equation*} where the summation is over all the $2$-sets $T$ in $I(H)$.
	If $H$ is simple, then the sum of all the $2\times 2$ principal minors of $\mathbf{U}(H)$ is $-\underset{e\in E(H)}{\sum} |\tau(e)|$. 
\end{thm}
	\begin{proof}
		Let $T=\{S_i, S_j\}$ be a $2$-set in $I(H)$. Since $H$ is loopless, the $2\times 2$ principal minor of $\mathbf{U}(H)$ corresponds to the rows and the columns indexed by $S_i$ and $S_j$ equals either $-(m(S_i\cup S_j))^2$ or zero according as $S_i\cup S_j$ is an edge in $H$ or not. So $\sqrt{|det(\mathbf{U}(H)[T|T])|}$ equals $m(e)$ only if $S\cup S'=e$ for some $e\in E^*(H)$.
		Thus, for an edge $e\in E^*(H)$, the sum of the square root of the modulus value of the $2\times2$ principal minor of $\mathbf{U}(H)$ corresponds to all the $2$-subset partitions of $e$ equals $|\tau(e)|$$m(e)$. 
		
		 Notice that if $H$	is simple, then the principal minor of $\mathbf{U}(H)$ corresponds to each $\{S_i, S_j\}\in \tau(e)$ equals $-1$ and all other $2\times 2$ principal minors are zero. This completes the proof.
		\end{proof}


\section{New structures and invariants of hypergraphs and unified eigenvalues}\label{sec5}
In this section, we introduce new structures and invariants on hypergraphs and we establish some results by relating them with the unified matrix of a hypergraph and its eigenvalues.

\begin{defn}\normalfont
Let $H$ be a hypergraph. An \textit{exact walk} in $H$ is a sequence $EW:(S=)S_0,$ $e_1,S_1,e_2,S_2,$ $\dots,e_{n-1},$ $S_{n-1},e_n,S_n(=S')$, where $\{S_{i-1}, S_i\}$ is a $2$-partition of the edge $e_i$ in $H$ for $i=1,2,\dots,n$. We say $EW$ joins $S$ and $S'$. Each $S_i$ occurs in $EW$ is a \textit{part} of $EW$. $S$ and $S'$ are \textit{the initial part} and \textit{the terminal part} of $EW$, respectively; Other parts are called \textit{internal parts} of $EW$. We call the set of all parts of $EW$ as its \textit{cover}. If $S=\{u\}$ and $S'=\{v\}$, then we say that $EW$ joins $u$ and $v$. The \textit{length of $EW$}, denoted by $l(EW)$, is the number of edges in $EW$, i.e., $n$.  
	We denote the exact walk simply by $EW:(S=)S_0,S_1,S_2,\dots,S_{n-1},S_n(=S')$, when the edges of the walk are self evident. 	
	Two exact walks $EW_1$ and $EW_2$ in $H$ are \textit{disjoint} if the set of all parts of $EW_1$ and the set of all parts of $EW_2$ are disjoint.
	
	An exact walk in $H$ is said to be an \textit{exact trail}, if $\{S_{i-1}, S_i\}\neq \{S_{j-1}, S_j\}$ for $i,j=1,2,\dots,n$ and $i \neq j$.	
	An exact walk in $H$ is said to be an \textit{exact path} if all its parts are distinct.
	Notice that in an exact path $EP:S_0,e_1,S_1,e_2,S_2,\dots,e_{n-1},S_{n-1},e_n,S_n$, an edge $e_i$ can be repeated as $e_j$ only when $\{S_{i-1}, S_i\}$ and $\{S_{j-1}, S_j\}$ are disjoint for $1\leq i<j\leq n$.	
	
	An exact walk in $H$ with at least three distinct edges is said to be an \textit{exact cycle} if its initial and terminal parts are the same and all internal parts are distinct. We say an \textit{exact cycle as odd (even)} if its length is odd (even).
We say an exact cycle of length three as an \textit{$e$-triangle}. 

Let $H$ be a loopless hypergraph. From the construction of $G_H$, it is clear that an exact walk (exact path) $S_0,S_1,S_2,\dots,S_{n-1},S_n$ in $H$ induces a walk (path) $S_0,S_1,S_2,\dots,S_{n-1},S_n$ in $G_H$ and vice versa. In a similar way,  an exact cycle $S_0,S_1,S_2,\dots,S_{n-1},S_0$ in $H$ induces a cycle $S_0,S_1,S_2,\dots,S_{n-1},S_0$ in $G_H$ and vice versa. 

\begin{example}\normalfont
	Consider the hypergraph $H$ shown in Figure~\ref{fig}. Then we have the following. 
	\begin{itemize}
		\item[(i)] $\{3\},\{4\},\{1,2\},\{4\},\{5,6\}$ is an exact walk joining $\{3\}$ and $\{5,6\}$ in $H$.
		\item[(ii)] 	$\{3\},\{1,2\},\{4\},\{5,6\}$ is an exact path joining $\{3\}$ and $\{5,6\}$ in $H$.
		\item[(iii)] $\{1,2\},\{3\},\{4\},\{1,2\}$ is an exact cycle in $H$.
	\end{itemize}
\end{example}

\begin{defn}\normalfont
	The \textit{exact girth} of a hypergraph $H$, denoted by $egr(H)$, is the minimum of the length of all exact cycles in $H$. 
	The \textit{odd exact girth} of $H$, denoted by $oegr(H)$, is defined as the minimum of the length of all odd exact cycles in $H$.
\end{defn}

An exact walk, an exact trail, an exact path and an exact cycle in a graph, respectively become a walk, a trail, a path and a cycle. Also, the exact girth of a graph is the same as the girth.

\begin{defn}\normalfont A hypergraph with at least two (three) distinct edges is said to be an \textit{unified path} (\textit{unified cycle}) if its edges and some of their parts can be arranged in an exact path (exact cycle) sequence whose parts are pairwise disjoint.
\end{defn}
Notice that an unified path (unified cycle) in $H$ is a path (cycle) in $H$.
\end{defn}

\begin{lemma}\normalfont\label{pathlem}
	Let $H$ be a unified path whose edges and some of their parts are arranged as an exact path $EP:S_0,e_1,S_1,e_2,S_2,\dots,e_{n-1},S_{n-1},e_n,S_n$, $n>1$ and $S_i$s are pairwise disjoint. Suppose the edges and some of their parts of $H$ are arranged in to another exact path $EP_1$, whose parts are pairwise disjoint. Then $EP_1$ must be the reverse of $EP$, i.e., $EP_1=S_n,e_n,S_{n-1},e_{n-1},S_{n-2},\dots,S_{1},e_1,S_0$.
\end{lemma}
 \begin{proof} 
 Since $H$ has $n$ edges, $EP_1$ must have exactly $n$ edges. 
 The edges in $EP_1$ must clearly follow one of the two sequences: either in the order $e_1, e_2,\dots,e_n$ or in the reverse, $e_n,e_{n-1},\dots,e_1$. 
 
 \textbf{Case~1.} If the edges of $EP_1$ comes as per the first order, then $EP_1$ is of the form $T_0,e_1,T_1,e_2,T_2,\dots,e_{n-1},T_{n-1},e_n,T_n$, where $\{T_{i-1}, T_i\}$ is a $2$-partition of the edge $e_i$ in $H$ for $i=1,2,3,\dots,n$.
 Since $n>1$, we have $T_1= e_1\cap e_2$. From $EP$, we have $e_1\cap e_2=S_1$. So, $T_1=S_1$. This implies that the complements of $S_1$ and $T_1$ in $e_1$ (resp. $e_2$) are the same. i.e., $T_0=S_0$ (resp. $T_2=S_2$).  In the same way, we can prove $T_i=S_i$ for all $i=3,4,\dots,n$.
 
\textbf{Case~2.} This case is similar to Case 1. In this case, we get $EP_1=S_n,e_n,S_{n-1},e_{n-1},$ $S_{n-2},\dots,S_{1},e_1,S_0$, the reverse sequence of $EP$. 	
 \end{proof}

As per Lemma~\ref{pathlem}, the edges and their parts of an unified path can be arranged exactly as two different exact paths whose parts are pairwise disjoint, which are reverse of each other. Conversely, an exact path of length at least two, whose parts are pairwise disjoint is simply the sequence of parts and edges of an unified path. Even though it is not a $1-1$ correspondence, we often specify an unified path by describing an associated exact path whose parts are pairwise disjoint and refer to that exact path as the unified path itself.

 In a similar way, if the edges and their parts of an unified cycle $H$ are arranged as an exact cycle $EC$ of length $n$ whose parts are pairwise disjoint, then the edges and their parts of $H$ are exactly arranged as $2n$ different exact cycles, each of which may start and end at any part of $EC$ and traverse in either direction. Conversely, an exact cycle whose parts are pairwise disjoint is simply the sequence of parts and edges of an unified cycle. Even though it is not a $1-1$ correspondence, we often specify an unified cycle by describing an associated exact cycle whose parts are pairwise disjoint and refer to that exact cycle as the unified cycle itself.

\begin{note}\normalfont
Two unified paths (unified cycles) of length $n$ need not be isomorphic. For instance, $EP_1:\{1\},\{2,3\},\{4,5\}$ is an unified path of length two joining $\{1\}$ and $\{4,5\}$, and $EP_2: \{1,2\}, \{3,4\}, \{5,6,7\}$ is an unified path of length two joining $\{1,2\}$ and $\{5,6,7\}$; 
But the cardinality of the vertex sets of these two unified paths are not equal.  So $EP_1$ and $EP_2$ are not isomorphic. In a similar way, we can get two non-isomorphic unified cycles of same length whose vertex sets are not equal.
\end{note}
	
	\begin{defn}\normalfont
		A hypergraph $H$ is said to be \textit{exactly connected} if for any two distinct vertices $u$ and $v$ in $H$, there exists $S, S'\in I(H)$ with $u\in S$ and $v\in S'$ such that $S$ and $S'$ are joined by an exact path in $H$.
	\end{defn}
\begin{example}\normalfont
	An exactly connected hypergraph $H$ is shown in Figure~\ref{fig2}.  
	\begin{figure}[ht]
		\begin{center}
			\includegraphics[scale=1.25]{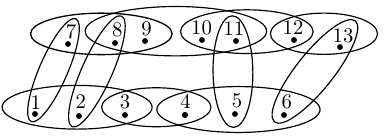}
		\end{center}\caption{An exactly connected hypergraph $H$}\label{fig2}
	\end{figure}
\end{example}
\begin{defn}\normalfont
	Let $H$ be a hypergraph. The \textit{exact distance} or simply the \textit{$e$-distance} between distinct vertices $u$ and $v$ in $H$, denoted by $ed_H(u,v)$, is the length of a shortest exact path joining $S, S'\in I(H)$ with $u\in S$ and $v\in S'$, if such an exact path exists; $\infty$, otherwise. For a vertex $u$ in $H$, we define $ed_H(u,u)=0.$ 
\end{defn}
\begin{defn}\normalfont
	The \textit{exact diameter} of an exactly connected hypergraph $H$ is the maximum of the $e$-distance between all the pairs of vertices in $H$, and we denoted it by $ED(H)$,   
	i.e., $ED(H)=\max\{ed_H(u,v)~|~u,v\in V(H)\}$.
\end{defn}

	It is clear that the exact distance defined on a hypergraph $H$ is a semi-metric. However, it is not a metric, as the triangle inequality may fail.
	For instance, in the hypergraph $H$ shown in Figure~\ref{fig2}, $ed_H(1,6)=3$ and $ed_H(6,13)=1$; but $ed_H(1,13)=5$.
	
	\begin{lemma}
	If $H$ is a loopless hypergraph, then 	the $(S_i, S_j)^{th}$ entry of $\mathbf{U}(H)^\ell$ is the number of exact walks in $H$ of length $\ell$ from $S_i$ to $S_j$.
	\end{lemma}
	\begin{proof} Since any exact walk in a loopless hypergraph $H$ induces a walk in its associated graph $G_H$ and $\mathbf{U}(H)=A(G_H)$, so the proof follows from the result \cite[p. 26]{bapat}: ``Let $G$ be a graph and let $v_i, v_j\in V(G)$. Then the $(v_i,v_j)^{th}$ entry of $A(G)^\ell$ equals the number of walks of length $\ell$ from $v_i$ to $v_j$ in $G$".
	\end{proof}

\begin{lemma}\label{diameter lemma}
	Let $H$ be a simple hypergraph with $e$-index $k$. If $u$ and $v$ are two vertices of $H$ such that $ed_H(u,v)=t$, then the matrices $I_k,\mathbf{U}(H), \mathbf{U}(H)^2,$ $\dots,\mathbf{U}(H)^t$ are linearly independent over $\mathbb{C}$. 
\end{lemma}
\begin{proof}
	Since $ed_H(u,v)=t,$ there exists $S_u, S_v\in I(H)$ with $u\in S_u$ and $v\in S_v$ such that $S_u$ and $S_v$ are joined by an exact path of length $t$ in $H$. Corresponding to this exact path, there exist a path joining $S_u$ and $S_v$ in $G_H$. We show that $d_{G_H}(S_u, S_v)=t$. Suppose not. Then there exists a path joining $S_u$ and $S_v$ in $G_H$ of length $t'<t$. Corresponding to this path we get an exact path joining $S_u$ and $S_v$ of length $t'$ which is less than $t$ in $H$. This contradicts the fact that $ed_H(u,v)=t$. 
	
	Suppose $G_H$ is connected. Then the result follows from \cite[Lemma~3.2]{bapat}. 
	
	Suppose $G_H$ is disconnected. Then $G_H$ is the disjoint union of $k(\geq2)$ components, say $G_1, G_2,\dots, G_k$, and so $A(G_H)$ is a block diagonal matrix whose $i$-th diagonal block is $A(G_i)$ for $i=1,2, \ldots, k$. 	
	Since there exists a path joining the vertices $S_u$ and $S_v$ in $G_H$, it follows that $S_u, S_v\in V(G_i)$ for some $i\in \{1,2,\dots,k\}$. Hence by \cite[Lemma~3.2]{bapat}, $I_{n_i}, A(G_i), A(G_i)^2,$ $\dots,A(G_i)^t$ are linearly independent over $\mathbb{C}$, where $n_i=|V(G_i)|$. It follows that $I_k, A(G_H),$ $A(G_H)^2, \dots, A(G_H)^t$ are linearly independent over $\mathbb{C}$. Since $A(G_H)=\mathbf{U}(H)$, the result follows.
\end{proof}
\begin{thm}
	If $H$ is an exactly connected hypergraph having $\ell$ distinct unified eigenvalues, then $ED(H)\leq\ell-1$.
\end{thm}
\begin{proof}
	By Lemma~\ref{diameter lemma}, $I, \mathbf{U}(H),$ $\dots, \mathbf{U}(H)^d$ are linearly independent, where $d=ED(H)$. Since, $H$ has $l$ distinct unified eigenvalues, the degree of the minimal polynomial of $\mathbf{U}(H)$ has degree $\ell$. It follows that there is no $m\geq \ell$ such that $I, \mathbf{U}(H),\dots,\mathbf{U}(H)^m$ are linearly independent. Thus $d<\ell$.
\end{proof}

\begin{thm}
Let $H$ be a hypergraph of $e$-index $k$ with eigenvalues $\lambda_1(H), \lambda_2(H), \dots, \lambda_k(H)$ such that $|\lambda_1(H)|\geq|\lambda_2(H)|\geq\dots\geq|\lambda_k(H)|$. Let $(x_1,x_2,\dots,x_k)$ be an orthonormal eigenvector corresponding to $\lambda_1(H)$ and let $\omega=\underset{i}{\min}\{x_i\}$. Then
	\begin{align*}
		ED(H)\leq \left\lceil \frac{\log{\frac{1-\omega^2}{\omega^2}}}{\log \frac{|\lambda_1(H)|}{|\lambda_2(H)|}} \right\rceil.
	\end{align*} 
\end{thm}
\begin{proof}
	It can be seen that $\{ed_H(u,v)~|~u,v\in V(H)\}\subseteq\{d_{G_H}(S,S')~|~S,S'\in I(H)\}$. Taking maximum on both sides, we have that $ED(H)$ is less than or equal to the diameter of $G_H$. So the result follows from~\cite[Theorem~2]{chung} and using the fact that $\mathbf{U}(H)=A(G_H)$.
\end{proof}
\begin{thm}
	Let $H$ be a simple hypergraph with $e$-index $k$. Let $P_{\mathbf{U}(H)}(x)=x^{k}+c_1x^{k-1}+\dots+c_k$. Then, we have the following.
	\begin{itemize}
		\item[(i)] The sum of all the $2\times2$ principal minors of $\mathbf{U}(H)=c_2=\frac{1}{2}(\partial(H)-\underset{S\in I(H)}{\sum}d_H(S))$.
		\item[(ii)] The sum of all the $3\times3$ principal minors of $\mathbf{U}(H)=-c_3=2t$, where $t$ is the number of $e$-triangles in $H$. 
	\end{itemize}
\end{thm}
\begin{proof}
	\begin{itemize}
		\item [(i)] It is clear that
		$c_2=\underset{1\leq i\leq j\leq k}{\sum}\lambda_i(H)\lambda_j(H)$ and $tr(\mathbf{U}(H))=0$. Applying these along with \eqref{L eq2} in the following identity, we get the result.
		\begin{align*}
			\big(\sum_{i=1}^{k}\lambda_i(H)\big)^2=\sum_{i=1}^{k}\lambda_i(H)^2+~2\sum_{1\leq i< j\leq k}\lambda_i(H)\lambda_j(H).
		\end{align*} 
		\item[(ii)] 
		It is known that $-c_3$ is the sum of all $3\times 3$ principal minors of $\mathbf{U}(H)$. Notice that corresponding to an $e$-triangle of $H$, there is a triangle in $G_H$ and vice versa. Also, $\mathbf{U}(H)=A(G_H)$. So, the proof follows from the result on graphs which states that for a simple graph $G$, the sum of all the $3\times3$ principal minors of $A(G)$ is twice the number of triangles in $G$~\cite[p.~26]{bapat}.
	\end{itemize}
\end{proof}	
In the next result, we characterize all the simple hypergraphs having no odd exact cycles in terms of its unified spectrum and its characteristic polynomial.
\begin{thm}
	Let $H$ be a simple hypergraph with $e$-index $k$. Then the following are equivalent.
	\begin{itemize}
		\item[(i)] $H$ has no odd exact cycle;
		\item[(ii)] If $P_{\mathbf{U}(H)}(x)=x^{k}+c_1x^{k-1}+\dots+c_{k-1}x+c_k$, then $c_{2r+1}=0$ for all $r=0,1,2,\dots$;
		\item[(iii)] The eigenvalues of $\mathbf{U}(H)$ are symmetric with respect to the origin, i.e., if $\lambda$ is an
		eigenvalue of $\mathbf{U}(H)$ with multiplicity $r$, then $-\lambda$ is also an eigenvalue of $\mathbf{U}(H)$ with
		the same multiplicity $r$.
	\end{itemize}
\end{thm}	
\begin{proof}
	It is clear that for each exact cycle of length $k$ in $H$, there exists a cycle of length $k$ in $G_H$ and vice versa. It follows that $H$ has no odd exact cycle if and only if $G_H$ has no odd cycle. Consequently, $G_H$ is bipartite if and only if $H$ has no odd exact cycle. So, the result follows from \cite[Theorem~3.14]{bapat} and from the fact that the spectrum of $\mathbf{U}(H)$ and $A(G_H)$ are the same.
\end{proof}


\section{Elementary $q$-subhypergraphs and the determinant of the unified matrix of a hypergraph}\label{sec6}
In this section, we introduce a special subhypergraph of a hypergraph, namely, an elementary $q$-subhypergraph. Then we study the interplay between these subhypergraphs and the invariants of the unified matrix such as its determinant, coefficients of its characteristic polynomial.

\begin{defn}\normalfont\label{elementary}
Let $H$ be a simple hypergraph with $e$-index $k$. For $q\in\{2,3,\dots,k\}$, a triplet $(H', \mathcal{C}(H'), \mathcal{E}(H'))_q$ is said to be an \emph{elementary $q$-subhypergraph of $H$} if the following conditions are satisfied:
\begin{itemize}
	\item[(i)] $H'$ is a subhypergraph of $H$;
	\item[(ii)] $\mathcal{C}(H')$ is a set, which is either empty or having some disjoint exact cycles $C_1, C_2,\dots,C_r$ in $H'$,
	 and $\mathcal{E}(H')$ is a set, which is either empty or having $2$-partitions of some edges $e_1, e_2,\dots,e_t$ in $H'$ such that $\underset{C\in \mathcal{C}(H')}{\sum}l(C)+2|\mathcal{E}(H')|=q$;
	\item[(iii)] Let $\mathcal{C}^*(H')$ denote the union of the covers of the disjoint exact cycles in $\mathcal{C}(H')$.  If $S\in \mathcal{C}^*(H')\cup \left(\underset{T\in \mathcal{E}(H')}{\bigcup}T\right)$, then $S$ belongs to either $\mathcal{C}^*(H')$ or at most one set in $\mathcal{E}(H')$.
\end{itemize}
\end{defn}
Notice that $\mathcal{C}(H')$ and $\mathcal{E}(H')$ cannot be empty simultaneously. For otherwise, $\underset{C\in \mathcal{C}(H')}{\sum}l(C)+2|\mathcal{E}(H')|\neq q$, since $q\geq 2$.
Further, notice that if $\mathcal{C}(H')\neq \Phi$, then $|\mathcal{C}(H')|=\sum_{i=1}^{r}l(C_i)$.

	As an edge in a graph $G$ has exactly one $2$-partition and an exact cycle of length $n$ in $G$ is a cycle of length $n$, it is clear that an elementary $q$-subhypergraph of $G$ becomes an elementary subgraph of $G$ on $q$ vertices.
\begin{example}\normalfont
	\begin{itemize}
		\item[(i)] 	Consider the hypergraph $H_1$ shown in Figure~\ref{fig3}. 	
Take $\mathcal{C}(H_1)=\{C_1\}$, where $C_1$ is the exact cycle $\{1\},\{2\},\{3\},\{1\}$. Then $\mathcal{C}^*(H_1)=\{\{1\},\{2\},\{3\}\}$. Take $\mathcal{E}(H_1)=\{\{\{4\},\{5,6\}\}, \{\{5\},\{4,6\}\}, \{\{6\},\{4,5\}\}\}$, which is the set of all $2$-partitions of the edge $e=\{4,5,6\}$.
		Notice that $\underset{C\in \mathcal{C}(H_1)}{\sum}l(C)+2|\mathcal{E}(H_1)|=9$.
		Also, any $S\in \mathcal{C}^*(H_1)\cup \left(\underset{T\in \mathcal{E}(H_1)}{\bigcup}T\right)$ belongs to either $\mathcal{C}^*(H_1)$ or at most one set in $\mathcal{E}(H_1)$. Thus, $(H_1,\mathcal{C}(H_1),\mathcal{E}(H_1))_9$ is an elementary $9$-subhypergraph of $H_1$. 
		
		\item[(ii)] 	Let $H_1'$ be the induced subhypergraph of $H_1$ induced by $\{1,2,3\}$. Take $\mathcal{C}(H_1')=\{C_1\}$, where $C_1$ is the exact cycle as given in part~(i). Take $\mathcal{E}(H_1')=\Phi$. Notice that  $\underset{C\in \mathcal{C}(H_1')}{\sum}l(C)+2|\mathcal{E}(H_1')|=3$. Thus, $(H_1',\mathcal{C}(H_1'),\mathcal{E}(H_1'))_3$ is an elementary $3$-subhypergraph of $H$.
		\begin{figure}[ht]
			\begin{center}
				\includegraphics[scale=1.2]{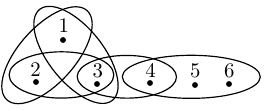}
			\end{center}\caption{The hypergraph $H_1$}\label{fig3}
		\end{figure}
	\end{itemize}
	\end{example}
\begin{example}\normalfont 
	Consider the hypergraph $H_2$ shown in Figure~\ref{fig1}. 
	Take $\mathcal{C}_1(H_2)=\{C_1, C_2\}$, where $C_1$ and $C_2$ are the exact cycles $\{1\},\{2,3\},$ $\{7\},\{8\},\{1\}$ and $\{3\},\{4,5\},\{6\},\{1,2\},\{3\}$, respectively.
	Then, $\mathcal{C}_1^*(H_2)=\{\{1\},\{3\},\{6\}, \{7\},\{8\},\{1,2\},\{2,3\},\{4,5\}\}$. Take $\mathcal{E}_1(H_2)=\{\{\{2\},\{1,3\}\}, \{\{4\},\{3,5\}\},$ $\{\{5\},\{3,4\}\}\}$, which is a set of $2$-partitions of the edges $e_1=\{1,2,3\}$ and $e_2=\{3,4,5\}$. Notice that $\underset{C\in \mathcal{C}_1(H_2)}{\sum}l(C)+2|\mathcal{E}(H_2)|=14$.
	Also, any $S\in \mathcal{C}_1^*(H_2)\cup \left(\underset{T\in \mathcal{E}(H_2)}{\bigcup}T\right)$ belongs to either $\mathcal{C}_1^*(H_2)$ or at most one set in $\mathcal{E}_1(H_2)$. Thus, $(H_2,\mathcal{C}_1(H_2),\mathcal{E}_1(H_2))_{14}$ is an elementary $14$-subhypergraph of $H_2$. 
	
	For the same hypergraph $H_2$, take $\mathcal{C}_2(H_2)=
	\{C\}$, where $C$ is the exact cycle $\{1\},\{8\},\{7\},\{2,3\},\{1\}$. Take $\mathcal{E}_2(H_2)=\{\{\{2\},\{3,7\}\},\{\{3\},\{1,2\}\},\{\{4\},\{3,5\}\},\{\{5\},\{3,4\}\},$ $\{\{6\},\{4,5\}\}\}$, which is a set of $2$-partitions of the edges $e_1$, $e_2$, $e_3=\{2,3,7\}$ and $e_4=\{3,4,5\}$. Then $(H_2,\mathcal{C}_2(H_2),\mathcal{E}_2(H_2))_{14}$ is an elementary $14$-subhypergraph of $H_2$.
	\begin{figure}[ht]
	\begin{center}
			\includegraphics[scale=1.3]{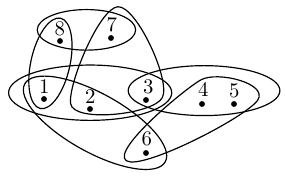}
		\end{center}\caption{The hypergraph $H_2$}\label{fig1}
	\end{figure}
\end{example}
\begin{lemma}\label{det lem}
	Let $H$ be a simple hypergraph with $e$-index $k$. For $t=2,3,\dots,k$, there is a one to one correspondence between the set of all elementary $t$-subhypergraphs of $H$ and the set of all elementary subgraphs of $G_H$ on $t$ vertices. 
\end{lemma}
\begin{proof}
	Let $X$ be the set of all elementary subgraphs of $G_H$ on $t$ vertices and let $Y$ be the set of elementary $t$-subhypergraphs of $H$. We shall show that there is a bijection from $X$ into $Y$.
	
		Let $G'\in X$ . We shall construct an elementary $t$-subhypergraph $(H',\mathcal{C}(H'), \mathcal{E}(H'))_t$ of $H$ using $G'$. 
			 For this, first we form the subhypergraph $H'$ of $H$ induced by all the edges $S\cup S'$, where $SS'\in E(G')$.
		 The method of choosing $\mathcal{C}(H')$ and $\mathcal{E}(H')$ is given below.
		\begin{itemize}
			\item If $G'$ has no cycle, then we take $\mathcal{C}(H')=\Phi$. 
			\item If $G'$ has a cycle of length $n$, then corresponding to this cycle there exists a unique exact cycle of length $n$ in $H$. So, we take $\mathcal{C}(H')$ as the set of all exact cycles of $H$ corresponding to all the cycles in $G'$. Notice that $\mathcal{C}(H')$ is a set of disjoint exact cycles in $H$, since all the cycles in $G'$ are pairwise disjoint.
			\item If $G'$ has no edge component, then we take $\mathcal{E}(H')=\Phi$.
			\item  If $G'$ has an edge component, say $SS'$, then $\{S, S'\}$ is a $2$-partition of an edge in $H'$. So, we take $\mathcal{E}(H')=\{\{S, S'\}~|~SS'~\text{is an edge component in}~G'\}$.
		\end{itemize} 
Note that $\underset{C\in\mathcal{C}(H')}{\sum}l(C)+2|\mathcal{E}(H')|=t$, since $G'$ has $t$ vertices. Also, $(H',\mathcal{C}(H'), \mathcal{E}(H'))_t$ satisfies the condition (iii)~of~Definition~\ref{elementary}, since each vertex in $G'$ belongs to at most one cycle or at most one edge. Thus $(H',\mathcal{C}(H'), \mathcal{E}(H'))_t\in Y$.

Now, by using the above construction, we define a map $f:X\to Y$ by $f(G')=(H',\mathcal{C}(H'), \mathcal{E}(H'))_t$ for all $G'\in X$.
	
It is clear that the constructions of $\mathcal{C}(H')$ and $\mathcal{E}(H')$ mentioned above depends only on the cycles and the edge components of $G'$. This shows the uniqueness of $(H',\mathcal{C}(H'), \mathcal{E}(H'))_t$ and so the map $f$ is well-defined.

To prove that the map $f$ is $1-1$. 
Let $G', G''\in X$ with $G\neq G'$. Let $(H',\mathcal{C}(H'), \mathcal{E}(H'))_t$ and $(H'',\mathcal{C}(H''), \mathcal{E}(H''))_t$ be the elementary $t$-subhypergraphs that are associated to $G'$ and $G''$ respectively, under $f$.
 Since $G'\neq G''$, they differ in at least one of their components. So, either $\mathcal{C}(H')\neq \mathcal{C}(H'')$ or $\mathcal{E}(H')\neq \mathcal{E}(H'')$ according as the cycle components or the edge components of $G'$ and $G''$ are different. Consequently, $(H',\mathcal{C}(H'), \mathcal{E}(H'))_t$ and $(H'',\mathcal{C}(H''), \mathcal{E}(H''))_t$ are distinct. So the map $f$ is $1-1$. 
 
Next, we show that the map $f$ is onto.
Let $(H',\mathcal{C}(H'), \mathcal{E}(H'))_t\in Y$. We shall construct an elementary subgraph  $G'$ of $G_H$ using $(H',\mathcal{C}(H'), \mathcal{E}(H'))_t$.

First we form two sets $A$ and $B$ as follows.
\begin{itemize}
	\item If $\mathcal{C}(H')=\Phi$, then we take $A=\Phi$.
	\item If $\mathcal{C}(H')\neq \Phi$, then corresponding to each exact cycle in $\mathcal{C}(H')$, we get a cycle in $G_H$. We take $A$ as the set of all such cycles in $G_H$.
	\item  If $\mathcal{E}(H')= \Phi$, then we take $B=\Phi$.
	\item If $\mathcal{E}(H')\neq \Phi$, then for each $2$-partition in $\mathcal{E}(H')$ corresponds to an edge in $G_H$. We take $B$ as the set of all such edges in $G_H$. 
\end{itemize}  

Since $C(H')$ and $E(H')$ cannot be empty simultaneously, it follows that $A$ and $B$ cannot be empty simultaneously.  Also, since $(H',\mathcal{C}(H'), \mathcal{E}(H'))_t$ satisfies the condition (iii)~of~Definition~\ref{elementary} and the exact cycles in $\mathcal{C}(H')$ are disjoint, we have the elements of $A$ and $B$ are mutually disjoint.
Now, we take $G'$ as the disjoint union of the disjoint cycles in $A$ and the disjoint edges in $B$. Then $G'$ has $t$ vertices, since $\underset{C\in \mathcal{C}(H')}{\sum}l(C)+2|\mathcal{E}(H')|=t$. Thus, $G'\in X$.

Also, notice that $f(G')=(H',\mathcal{C}(H'), \mathcal{E}(H'))_t$. Therefore, $f$ is onto. This completes the proof.
\end{proof}
\begin{notation}\normalfont
Let $H$ be a simple hypergraph. Let $(H',\mathcal{C}(H'), \mathcal{E}(H'))_t$ be an elementary $t$-subhypergraph of $H$. Then we denote $c(H'):=|\mathcal{C}(H')|$ and $e(H'):=|\mathcal{E}(H')|$.

It is clear that if $H$ is a simple graph and $H'$ is its elementary subgraph, then $c(H')$ is the number of cycle components in $H'$ and $e(H')$ is the number of edge components in $H'$. 
\end{notation}

In the next result, we express the determinant of the unified matrix of a hypergraph using some parameters of its elementary $k$-subhypergraphs.
\begin{thm}
	Let $H$ be a simple hypergraph with $e$-index $k$. Then,
	\begin{equation}\label{det eq}
	det(\mathbf{U}(H))=\displaystyle\sum(-1)^{k-c(H')-e(H')}2^{c(H')},
	\end{equation} where the summation is over all elementary $k$-subhypergraphs $(H',\mathcal{C}(H'), \mathcal{E}(H'))_k$ of $H$.
\end{thm}
\begin{proof}
Consider an elementary subgraph $G'$ of $G_H$.  By Lemma~\ref{det lem}, there exists a unique elementary $k$-subhypergraph $H'$ of $H$ corresponding to $G'$.
	Also, notice that $c(G')=c(H')$ and $e(G')=e(H')$. So, the proof follows from the result~\cite[Theorem~3.8]{bapat}: ``For a graph $G$ with $n$ vertices,
	$det(A(G))=\displaystyle\underset{G'}{\sum}(-1)^{n-c(G')-e(G')}2^{c(G')}$, where the summation is over all elementary spanning subgraphs $G'$ of $G$". 
\end{proof}
\begin{example}\normalfont
	Let $H$ be a hypergraph shown in Figure~\ref{Det}. Clearly, the $e$-index of $H$ is $10$. 
		\begin{figure}[ht]
		\begin{center}
			\includegraphics[scale=1.3]{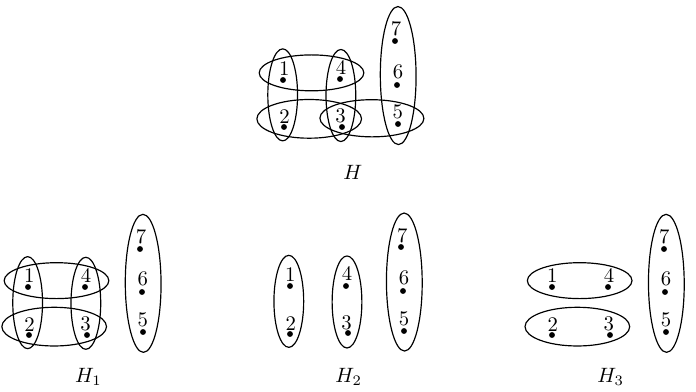}
		\end{center}\caption{The hypergraph $H$ and its elementary $10$-subhypergraphs $H_1$, $H_2$, $H_3$.}\label{Det}
	\end{figure}
	The elementary $10$-subhypergraphs of $H$ are the following:
	\begin{itemize}
		\item[(i)] $(H_1, \mathcal{C}(H_1), \mathcal{E}(H_1))_{10}$, where $H_1$ is shown in Figure~\ref{Det}, $\mathcal{C}(H_1)=\{\{1\},\{2\},\{3\},\{4\}\}$ and $\mathcal{E}(H_1)=\tau(\{5,6,7\})$; 
		\item[(ii)] $(H_2, \mathcal{C}(H_2), \mathcal{E}(H_2))_{10}$, where $H_2$ is shown in Figure~\ref{Det}, $\mathcal{C}(H_2)=\Phi$ and $\mathcal{E}(H_2)=\tau(\{5,6,7\})\cup \tau(\{1,2\})\cup \tau(\{3,4\})$;
		\item[(iii)] $(H_3, \mathcal{C}(H_3), \mathcal{E}(H_3))_{10}$, where $H_3$ is shown in Figure~\ref{Det}, $\mathcal{C}(H_3)=\Phi$ and $\mathcal{E}(H_3)=\tau(\{5,6,7\})\cup \tau(\{1,4\})\cup \tau(\{2,3\})$.
	\end{itemize}

 From the formula~\eqref{det eq}, $det(\mathbf{U}(H))=(-1)^{10-1-3}(2)+(-1)^{10-0-5}(2)^0+(-1)^{10-0-5}(2)^0=0$. This fact is also evident since $\mathbf{U}(H)$ has linearly dependent columns.
\end{example}
\begin{thm}
	Let $H$ be a simple hypergraph with $e$-index $k$ and let $P_{\mathbf{U}(H)}(x)=x^{k}+c_1x^{k-1}+\dots+c_k$. Then for $t\in\{1,2,\dots,k\}$,
	\begin{align*} c_t=\displaystyle\sum(-1)^{c(H')+e(H')}2^{c(H')},
	\end{align*} where the summation is over all elementary $t$-subhypergraphs $(H',\mathcal{C}(H'), \mathcal{E}(H'))_t$ of $H$.
\end{thm}
\begin{proof}	
Consider an elementary $t$-subhypergraph $(H',\mathcal{C}(H'), \mathcal{E}(H'))_t$ of $H$. Correspondingly, there exists an elementary subgraph $G'$ of $G$ on $t$ vertices, by Lemma~\ref{det lem}. Since, there is a one-to-one correspondence between $\mathcal{C}(H')$ and the set of all the cycle components of $G'$, it follows that $c(G')=c(H')$. Similarly, we have $e(G')=e(H')$. Hence the proof follows from \cite[Theorem~3.10]{bapat}, since $\mathbf{U}(H)=A(G_H)$. Note that $c_1=0$, since $H$ has no loops.
\end{proof}
\begin{thm}
		Let $H$ be a simple hypergraph with $e$-index $k$ and let $P_{\mathbf{U}(H)}(x)=x^{k}+c_1x^{k-1}+\dots+c_{k-1}x+c_{k}$.
	\begin{itemize}
		\item[(i)]The odd exact girth $oegr(H)$ is equal to the index of the first non-zero coefficient from the sequence $c_1, c_3, c_5,\dots$. Moreover, the number of exact cycles of length $d$ in $H$ is $-\frac{1}{2}c_d,$ where $d=oegr(H)$.
		
		\item[(ii)]The exact girth $egr(H)$ is equal to the index of the first non-zero coefficient from the sequence ${\gamma_1}, {\gamma_2},{\gamma_3}\dots$, where for $i=1,2,\dots,k,$ 
		 \begin{align*}\label{c_i}
		{\gamma_i}=\begin{cases}
		c_i, & \text{if}~ i~ \text{is odd};\\
		c_i-(-1)^{i/2}\omega_{i}, & \text{if}~ i~ \text{is even},
		\end{cases}
		\end{align*}
		and $\omega_{i}$ is the number of elementary $i$-subhypergraphs $(H',\Phi, \mathcal{E}(H'))_i$ of $H$.
		Also, the number of exact cycles of length $t$ in $H$ is $-\frac{1}{ 2}{\gamma_t}$, where $t=egr(H)$. 
	\end{itemize}
\end{thm}
	\begin{proof}
		It is clear that each exact cycle of length $k$ in $H$ corresponds to a cycle of length $k$ in $G_H$ and vice versa. Since $\mathbf{U}(H)=A(G_H)$, the result follows from~\cite[Theorems~3.2.5 and 3.2.6]{cvetko}.
	\end{proof}


\section{Unified spectrum of unified cycles and unified paths }\label{sec7}

In this section, we determine the unified spectrum of unified cycles and unified paths.

An unified cycle $EC$ and an unified path $EP$ of length $n$ whose parts are of cardinality one are respectively the cycle graph and the path graph of length $n$. Since $\mathbf{U}(EC)=A(EC)$ and $\mathbf{U}(EP)=A(EP)$, the unified eigenvalues of $EC$ are $2~ \cos(\frac{2\pi k}{n})$ for $k=1,2,\dots,n$ and the unified eigenvalues of $EP$ are $2~ \cos(\frac{\pi k}{n+2})$ for $k=1,2,\dots,n+1$ (c.f.~\cite[Theorems~3.7 and 3.6]{bapat}).

So, in the following results, we determine the unified spectrum of unified cycles and unified paths which are not graphs.
\begin{thm}\label{Cn thm}
	Let $EC:S_0,S_1,S_2,\dots,S_{n-1},S_n(=S_0)$ be an unified cycle of length $n\geq 3$ with $|S_i|>1$ for some $i\in \{1,2,\dots,n\}$. Then the unified eigenvalues of $EC$ are
\begin{itemize}
		\item[(i)] $2\cos(\frac{2\pi k}{n})$ with multiplicity $1$ for $k=1,2,\dots,n$;
		\item[(ii)]$\pm1$ with multiplicity $t_1$, provided $|S_{i-1}|$ and $|S_i|>1$ for at least one $i\in \{1,2,\dots n\}$, 	where
		 \begin{equation}\label{t1}
		t_1=\underset{i=1}{\overset{n}{\Sigma}}p_i
		\end{equation} and for $i=1,2,\dots,n$,
		\begin{equation}\label{del}
		{p_i}=\begin{cases}
		0, & \text{if}~|S_{i-1}|~\text{or}~|S_i|=1,\\
		\frac{1}{2}\underset{1\leq l_2<|S_i|}{\underset{1\leq l_1<|S_{i-1}|;}{\sum}}\binom{|S_{i-1}|}{l_1}\binom{|S_i|}{l_2}, & \text{otherwise};
		\end{cases}
		\end{equation} 
		\item[(iii)] 		
		$0,~\pm\sqrt{2}$ with multiplicity $t_2$, provided $\underset{e\in E(EC)}{\sum}|\tau(e)|\neq n+t_1$, where $$t_2=\frac{1}{2}\left[\left(\underset{e\in E(EC)}{\sum}|\tau(e)|\right)-n-t_1\right].$$
	\end{itemize}
\end{thm}
\begin{proof}
	For each $i=1,2,\dots, n$, let $e_i$ be the edge in $EC$ having $\{S_{i-1}, S_i\}$ as its two partition.
	Let $G_{EC}$ be the associated graph of the unified cycle $EC$. Corresponding to $EC$, there exists a unique cycle $C$ in $G_{EC}$ whose vertices are $S_0,S_1,S_2,\dots S_{n-1},S_0$. 
	We shall show that $G_{EC}$ is the disjoint union of $C$, some copies of $K_2$, and some copies of $P_3$.

Since $|S_i|>1$ for some $i\in \{1,2,\dots,n\}$, it can be easily seen that there exists a vertex in in $G_{EC}$ which does not belong to the cycle $C$.
Now, we consider such a vertex $T$ in $G_{EC}$.  Let $T\subset e_i(=S_{i-1}\cup S_{i})$ for some $i\in\{1,2,\dots,n\}$. Then we have the following two cases.
 
 \textbf{Case~1.}  $T\subseteq S_i$.
 
 Since $T$ does not belong to the cycle $C$, $T\neq S_i$. Therefore, $T\subset S_i$ and so $|S_i|>1$. 
 
 Notice that the vertices $T$ and $e_{i}\backslash T$ form an edge in $G_{EC}$. Similarly, $T$ and $e_{i+1}\backslash T$ form an edge in $G_{EC}$. As there is no edge $e~(\neq e_i,e_{i+1})$ in $EC$ such that $T\subset e$, it turns out that no vertex in $G_{EC}$ is adjacent to $T$ other than $e_{i}\backslash T$ and $e_{i+1}\backslash T$.		
 As there is no edge $e~(\neq e_i, e_{i+1})$ in $EC$ such that $e_i\backslash T\subset e$ and $e_{i+1}\backslash T\subset e$, it follows that $e_i\backslash T$ and $e_{i+1}\backslash T$ are adjacent only with $T$. 	
 Therefore, the vertices $T$, $e_{i}\backslash T$ and $e_{i+1}\backslash T$ forms a $P_3$ as a component in $G_{EC}$ with middle vertex $T$.

\textbf{Case~2.} $T\nsubseteq S_{i-1}$ and $S_{i}$.

\textbf{Subcase~2a.} Suppose $e_i\backslash T\subseteq S_{i-1}$ or $S_{i}$.

Since $T\neq S_{i-1}$ or $S_i$, we have $e_i\backslash T\subset S_{i-1}$ or $S_{i}$.
 Therefore, by Case~1, $e_i\backslash T$, $T$, $e_{i+1}\backslash(e_i\backslash T)$ forms $P_3$ component in $G_{EC}$. Thus, $T$ belongs to a $P_3$ component in $G_{EC}$.
 
	\textbf{Subcase~2b.}
Suppose $e_i\backslash T\nsubseteq S_{i-1}$ and $S_{i}$.

 Notice that there is an edge joining the vertices $T$, $e_i\backslash T$ in $G_{EC}$. As $T\nsubseteq S_{i-1}$ and $T\nsubseteq S_{i}$, there is no edge $e~(\neq e_i)$ in $EC$ such that $T\subset e$. This implies that no vertex in $G_{EC}$ is adjacent to $T$ except $e_i\backslash T$. In a similar way, as $e_i\backslash T \nsubseteq S_{i-1}$ and $e_i\backslash T \nsubseteq S_{i}$, there is no edge $e~(\neq e_i)$ in $EC$ such that $e_i\backslash T\subset e$. This implies that no vertex in $G_{EC}$ is adjacent to $e_i\backslash T$ except $T$. Thus, $T$ and $e_i\backslash T$ forms a component $K_2$ in $G_{EC}$.

 Suppose $|S_{i-1}|$ or $|S_i|=1$. Without loss of generality, we assume that $|S_i|=1$. Since $T\nsubseteq S_{i-1}$ and $\neq S_{i}$, we have $S_i\subset T$. It turns out that $e_i\backslash T\subset S_{i-1}$, a contradiction to Subcase~b. So, $|S_{i-1}|$ and $|S_i|>1$. Therefore, from the above two subcases, it is clear that if $|S_{i-1}|$ or $|S_i|=1$, then no $T\subset e_i$ belongs to a $K_2$ component in $G_{EC}$.
 
 Since each $e_i$ has $2p_i$ number of such $T$s, there exists $p_i$ disjoint copies of $K_2$ in $G_{EC}$, where $p_i$ is given in $\eqref{del}$. Totally, we have $t_1=\underset{i=1}{\overset{n}{\Sigma}}p_i$ disjoint copies of $K_2$ in $G_{EC}$, provided $|S_{i-1}|$ and $|S_i|>1$ for at least one $i\in \{1,2,\dots n\}$.
	
From the above two cases, any $S\in V(G_{EC})$ which does not belong to $C$, belongs to either a copy of $K_2$ or a copy of $P_3$, i.e., $G_{EC}$ is the disjoint union of $C$, some copies of $K_2$ and some copies of $P_3$.

 As there are $t_1$ copies of $K_2$ and one copy of $C$, there are $\underset{e\in E(EC)}{\sum}|\tau(e)|-n-t_1$ edges remaining in $G_{EC}$ only if $\underset{e\in E(EC)}{\sum}|\tau(e)|\neq n+t_1$.
So there are $$\frac{1}{2}\left[\left(\underset{e\in E(EC)}{\sum}|\tau(e)|\right)-n-t_1\right]$$ copies of $P_3$ in $G_{EC}$, provided $\underset{e\in E(EC)}{\sum}|\tau(e)|\neq n+t_1$.
	
The eigenvalues of $K_2~(=P_2)$, $P_3$ and $C_n$ are mentioned above.
	 
Thus the proof follows from the fact that the spectrum of the adjacency matrix of the disjoint union of $r$ graphs is the union of the spectra of their adjacency matrices. 
\end{proof}
\begin{cor}
	Let $EC: S_0,S_1,S_2,\dots,S_{n-1},S_0$ be an $2r$-uniform unified cycle of length $n \geq 3$ with $|S_i|=r$~$(>1)$ for all $i=0,1,\dots,n-1$. Then the unified eigenvalues of $EC$ are \begin{itemize}
		\item[(i)] $2~ \cos(\frac{2\pi k}{n})$ with multiplicity $1$ for $k=1,2,\dots,n$;
		\item[(ii)]$\pm1$ with multiplicity $n\delta_1$, where 
		\begin{equation}\label{del1}
		{\delta_1}=
	\frac{1}{2}\underset{1\leq l_1,l_2<r}{\sum}\binom{r}{l_1}\binom{r}{l_2}.
		\end{equation} 
		\item[(iii)] 		
		$0,~\pm\sqrt{2}$ with multiplicity $\delta_2$, where
		\begin{equation}\label{del2}
	{\delta_2}=
		\frac{n}{2}\left[\binom{2r}{1}+\binom{2r}{2}+\dots+\binom{2r}{r-1}+\frac{1}{2}\binom{2r}{r}-(\delta_1+1)\right].
		\end{equation}
	\end{itemize}
\end{cor}
\begin{proof}
	Since $|S_i|=r$~($>1)$, each edge in $EC$ is of cardinality $2r$, so we have 
	\begin{equation}\label{tau1}
		\underset{e\in E(EC)}{\sum}|\tau(e)|=
			n\left[\binom{2r}{1}+\binom{2r}{2}+\dots+\binom{2r}{r-1}+\frac{1}{2}\binom{2r}{r}\right].
	\end{equation}  
Since $|S_{i-1}|$ and $|S_i|>1$, we have $p_i=\delta_1$ for each $i=1,\dots,n$, where $\delta_1$ is given in~\eqref{del1}. Now, by substituting  $p_i=\delta_1$ for $i=1,2,\dots,n$ and \eqref{tau1} in Theorem~\ref{Cn thm}, we get the result. 
\end{proof}
\begin{thm}\label{Pn thm}
	Let $EP:S_0,S_1,S_2,\dots,S_{n-1},S_n$ be an unified path of length $n \geq 2$ with $|S_i|>1$ for some $i\in \{0,1,\dots,n\}$. Then the unified eigenvalues of $EP$ are \begin{itemize}
		\item[(i)]$2\cos(\frac{\pi k}{n+2})$  with multiplicity $1$ for $k=1,2,\dots,n+1$;
		\item[(ii)]$\pm1$ with multiplicity
		$$t=\begin{cases}
		t_1, &\text{if}~|S_0|=|S_n|=1;\\
		t_1+2^{|S_n|}-2, &\text{if}~|S_0|=1~\text{and}~|S_n|>1;\\
		t_1+2^{|S_0|}-2, &\text{if}~|S_n|=1~\text{and}~ |S_0|>1;\\
		t_1+2^{|S_0|}+2^{|S_n|}-4, &\text{if}~|S_0|>1~\text{and}~|S_n|>1,
		\end{cases}$$	
		
		 provided $|S_{i-1}|$ and $|S_i|>1$ for at least one $i\in \{1,2,\dots n\}$, where $t_1$ is given in~\eqref{t1};
		\item[(iii)] $0$ and $\pm\sqrt{2}$ with multiplicity \begin{equation}\label{t'}
		t'=\frac{1}{2}\left[\left(\underset{e\in E(EP)}{\sum}|\tau(e)|\right)-n-t\right],
		\end{equation}
	provided $\underset{e\in E(EP)}{\sum}|\tau(e)|\neq n+t$.
	\end{itemize}
\end{thm}
\begin{proof}
		For each $i=1,2,\dots, n$, let $e_i$ be the edge in $EC$ having $\{S_{i-1}, S_i\}$ as its two partition.
	Let $G_{EP}$ be the associated graph of the unified path $EP$. Corresponding to $EP$, there is a unique path $P$ in $G_{EP}$ whose vertices are $S_0,S_1,S_2,\dots S_{n-1},S_n$.
	
	Now, similar to the Case~2 given in the proof of Theorem~\ref{Cn thm}, firstly we have $t_1$ copies of $K_2$ in $G_{EP}$, where $t_1$ is given in~\eqref{t1}, provided $|S_{i-1}|$ and $|S_i|>1$ for at least one $i\in \{1,2,\dots n\}$. 
	
	Notice that $S_0$ and $S_n$ do not form an edge in $EP$. Therefore,
	\begin{itemize}
		\item if $|S_0|=|S_n|=1$, then there are only $t_1$ copies of $K_2$ in $G_{EP}$;
		\item if $|S_0|>1$ and $|S_n|>1$, then for each proper non-empty subset $T$ of $S_0$ (resp. $S_n$), $T$ and $e_1\backslash T$ (resp. $T$ and $e_n\backslash T$) are adjacent in $G_{EP}$, and so we get a copy of $K_2$ in $G_{EP}$.
		Thus, we get $2^{|S_0|}+2^{|S_n|}-4$ copies of $K_2$, in addition to the $t_1$ copies of $K_2$ in $G_{EP}$; 
		\item if $|S_0|=1$ and $|S_n|>1$, then we have $2^{|S_n|}-2$ copies of $K_2$, in addition to the $t_1$ copies of $K_2$ in $G_{EP}$;
    	\item if $|S_n|=1$ and $|S_0|>1$, then we have $2^{|S_0|}-2$ copies of $K_2$, in addition to the $t_1$ copies of $K_2$ in $G_{EP}$.
	\end{itemize} 
	On the other hand, from Case~1 given in the proof of Theorem~\ref{Cn thm}, for any proper non-empty subset $T$ of $S_i$ ($i\neq 0,n$), we have $t'$ copies of $P_3$ in $G_{EP}$, where $t'$ is given in~$\eqref{t'}$, provided $\underset{e\in E(EP)}{\sum}|\tau(e)|\neq n+t$.
	
	Thus, any $S\in V(G_{EP})$ which does not belong to the path $P$, belongs to either a copy of $K_2$ or a copy of $P_3$, i.e., $G_{EP}$ is the disjoint union of $P$, $t$ copies of $K_2$ and $t'$ copies of $P_3$. 
	
	The rest of the proof follows similarly to that of Theorem~\ref{Cn thm}.
\end{proof}
\begin{cor}
	Let $EP: S_0,S_1,S_2,\ldots,S_{n-1},S_n$ be an $2r$-uniform unified path of length $n \geq 2$ with $|S_i|=r$~$(>1)$ for all $i=0,1,\ldots,n$. Then its unified eigenvalues are
	 \begin{itemize}
		\item[(i)] $2~ \cos(\frac{\pi k}{n+2})$  with multiplicity $1$ for $k=1,2,\dots,n+1$;
		\item[(ii)]$\pm1$ with multiplicity $\delta$, where
			$\delta=n\delta_1+2^{r+1}-4$, and $\delta_1$ is given in~\eqref{del1};
		\item[(iii)] 	 $0,~\pm\sqrt{2}$ with multiplicity 
			$${\delta^*}=
		\frac{1}{2}\left[n\left\{ \binom{2r}{1}+\binom{2r}{2}+\dots+\binom{2r}{r-1}+\frac{1}{2}\binom{2r}{r}\right\}-(\delta+n)\right].$$
	\end{itemize}
\end{cor}
\begin{proof}
	Since $|S_i|=r$~($>1)$, each edge in $EC$ is of cardinality $2r$, so we have  \begin{equation}\label{tau}
	\underset{e\in E(EP)}{\sum}|\tau(e)|=
	n\left[\binom{2r}{1}+\binom{2r}{2}+\dots+\binom{2r}{r-1}+\frac{1}{2}\binom{2r}{r}\right].
	\end{equation} 
Since $|S_{i-1}|$ and $|S_i|>1$, we have $p_i=\delta_1$ for each $i=1,\dots,n$, where $\delta_1$ is given in~\eqref{del1}. Now, by substituting $t_1=n\delta_1$ and \eqref{tau} in Theorem~\ref{Pn thm}, we get the result.
\end{proof}


\section{Unified characteristic polynomial of the hypergraphs constructed by some hypergraph operations}\label{sec8}

In this section, we determine the unified characteristic polynomial of a hypergraph constructed by some hypergraph operations.

	Let $H_i$ be a hypergraph on $n_i$ vertices for $i=1,2,\dots,k$. Then it is straight forward to see that the unified characteristic polynomial of the disjoint union of $H_is$ is
	\[P_{\mathbf{U}({\overset{.}{\bigcup}H_i})}(x)=\overset{k}{\underset{i=1}{\prod}}P_{\mathbf{U}(H_i)}(x).\]

\begin{defn}\normalfont
	 A vertex $v$ in a hypergraph $H$ is said to be a \textit{$g$-vertex} if $v\notin e$ for any $e\in E(H)$ with $|e|>2$.
\end{defn}
\begin{thm}
	Let $H$ be a simple hypergraph.
	Let $u$ be a vertex in $H$, and let $H_{uv}$ denote the hypergraph obtained from $H$ by adding a new vertex $v$ and joining it with $u$. Then the characteristic polynomial of $\mathbf{U}(H_{uv})$ is $$P_{\mathbf{U}(H_{uv})}(x)=xP_{\mathbf{U}(H)}(x)-P_{\mathbf{U}(H)(\{u\}|\{u\})}(x).$$
	Furthermore, if $u$ is a $g$-vertex, then
	$$P_{\mathbf{U}(H_{uv})}(x)=x P_{\mathbf{U}(H)}(x)-P_{\mathbf{U}(H\backslash\{u\})}(x).$$
\end{thm}
\begin{proof}
	The associated graph $G_{H_{uv}}$ of $H_{uv}$ is the same as the graph $(G_H)_{\{u\}\{v\}}$ which is obtained from $G_H$ by adding the pendant edge $\{\{u\},\{v\}\}$ at the vertex $\{u\}$. Notice that $A((G_H)_{\{u\}\{v\}})=A(G_{H_{uv}})=\mathbf{U}(H_{uv})$ and $A(G_H)=\mathbf{U}(H)$.
	So, the proof directly follows from~\cite[Theorem~2.2.1]{cvetko}: ``Let $G$ be a simple graph with vertices $v_1$, $v_2$, \dots, $v_n$. Let $G_j$ denote the graph obtained from $G$ by adding a pendant edge at the vertex $v_j$, then $P_{A(G_j)}(x)=xP_{A(G)}(x)-P_{A(G\backslash\{v_j\})}(x)".$ This completes the proof of the first part.
	
	Furthermore, if $u$ is a $g$-vertex in $H$, then the unified matrix of $H\backslash\{u\}$ is $\mathbf{U}(H)(\{u\}|\{u\})$. This completes the proof.
\end{proof}

Let $H_1$ and $H_2$ be two vertex-disjoint hypergraphs, and let $u\in V(H_1)$, $v\in V(H_2)$.
The \textit{coalescence} of $H_1$ and $H_2$ with respect to $u,v$, denoted by $H_1(u)\diamond H_2(v)$, is the hypergraph
obtained from $H_1, H_2$ by identifying $u$ with $v$ \cite{fan2020}.

\begin{thm}
	Let $H_1$ and $H_2$ be vertex-disjoint simple hypergraphs, and let $u\in V(H_1)$, $v\in V(H_2)$. Then, 
	\begin{eqnarray*}
		P_{\mathbf{U}(H_1(u)\diamond H_2(v))}(x)= P_{\mathbf{U}(H_1)}(x) P_{\mathbf{U}(H_2)(\{v\}|\{v\})}(x)+P_{\mathbf{U}(H_1)(\{u\}|\{u\})}(x)P_{\mathbf{U}(H_2)}(x)-\\
		xP_{{\mathbf{U}(H_1)}(\{u\}|\{u\})}(x) P_{\mathbf{U}(H_2)(\{v\}|\{v\})}(x).
	\end{eqnarray*}
	In particular, if  $u$ and $v$ are $g$-\textit{vertices}, then
	\begin{equation*}
	P_{\mathbf{U}(H_1(u)\diamond H_2(v))}(x)=P_{\mathbf{U}(H_1)}(x) P_{\mathbf{U}(H_2\backslash \{v\})}(x)+P_{\mathbf{U}(H_1\backslash\{u\})}(x) P_{\mathbf{U}(H_2)}(x)-xP_{\mathbf{U}(H_1\backslash\{u\})}(x)P_{\mathbf{U}(H_2\backslash\{v\})}(x).
	\end{equation*}
\end{thm}
\begin{proof}
	Let $G_{H_1}$ and $G_{H_2}$ be the associated graphs of $H_1$ and $H_2$, respectively. It can be easily seen that the associated graph corresponding to the hypergraph $H_1(u)\diamond H_2(v)$ is isomorphic to the graph $G_{H_1}(\{u\})\diamond G_{H_2}(\{v\})$. Notice that  $A(G_{H_1\backslash\{u\}})={\mathbf{U}(H_1)}(\{u\}|\{u\})$ and $A(G_{H_2\backslash\{v\}})={\mathbf{U}(H_2)}(\{v\}|\{v\})$. Thus, the result follows from~\cite[Theorem~2.2.3]{cvetko}: ``Let $G$ and $H$ be simple graphs and let $G(u)\diamond H(v)$ be the coalescence in which the vertex $u$ of $G$ is
	identified with the vertex $v$ of $H$. Then $P_{A(G(u)\diamond H(v))}(x)= P_{A(G)}(x) P_{A(H\backslash\{v\})}(x)+P_{A(G\backslash\{u\})}(x) P_{A(H)}(x)-xP_{A(G\backslash\{u\})}(x) P_{A(H\backslash\{v\})}(x)."$
	
	In particular, if $H_1(u)\diamond H_2(v)$ is obtained by identifying the $g$-vertices, then $\mathbf{U}(H_2)(\{v\}|\{v\})=\mathbf{U}(H_2\backslash \{v\})$ and $\mathbf{U}(H_1)(\{u\}|\{u\})= \mathbf{U}(H_1\backslash \{u\})$. Thus, we get the result as desired.
\end{proof}

\begin{thm}
	Let $H$ be a simple hypergraph. Let $u,v\in V(H)$. Let $H'$ denote the hypergraph obtained from $H$ by adding a set $S$ of finite number of new vertices to $H$, and making the elements of the set $S$ together with the vertices $u$ and $v$ as an edge in $H'$. Then the unified characteristic polynomial of $H'$ is given as follows:
	\begin{itemize}
		\item [(i)] If there is an edge $e\in E(H)$ with $|e|>2$ such that $u,v\in e$, then \begin{align*}
		P_{\mathbf{U}(H')}(x)&=(x^2-1)^{\alpha_1}\left(x^3P_{\mathbf{U}(H)}(x)-x^2P_{\mathbf{U}(H)(\{u\}|\{u\})}(x)-(x^2-x)P_{\mathbf{U}(H)(\{v\}|\{v\})}(x)\right.\nonumber\\& 
		\left.+(x-1)P_{\mathbf{U}(H)(T|T)}(x)-x^2P_{\mathbf{U}(H)(\{u,v\}|\{u,v\})}(x)+xP_{\mathbf{U}(H)(T'|T')}(x)\right),
		\end{align*} where $T=\{\{u\},\{v\}\}$, $T'=\{\{u\},\{u,v\}\}$ and \begin{equation}\label{alpha1}
		\alpha_1=\begin{cases}
		2^{|S|+1}-4, \qquad\text{if}~ |S|>1;\\
		~~~0, \qquad~~\text{otherwise}.
		\end{cases}
		\end{equation}
		\item[(ii)] If there is no edge $e\in E(H)$ with $|e|>2$ such that $u,v\in e$, then \begin{align*}
		P_{\mathbf{U}(H')}(x)=(x^2-1)^{\alpha_2}\left(x^2P_{\mathbf{U}(H)}(x)-xP_{\mathbf{U}(H)(\{u\}|\{u\})}(x)-xP_{\mathbf{U}(H)(\{v\}|\{v\})}(x)+P_{\mathbf{U}(H)(T|T)}\right),
		\end{align*}
	\end{itemize}
	where 
	\begin{align}\label{alpha2}
	\alpha_2=\begin{cases}
	2^{|S|+1}-3, \qquad\text{if}~ |S|>1;\\
	~~~0, \qquad~~\text{otherwise}.
	\end{cases}
	\end{align}
\end{thm}
\begin{proof}
	Let $X=S\cup\{u,v\}$. 
		First we shall prove part (i). Suppose that there is an edge $e\in E(H)$ with $|e|>2$ such that $u,v\in e$.
		Then $\{u,v\}\in I(H)$. Notice that $E(H')$ contains the edge $X$ in addition to the edges in $E(H)$. We know that $\mathbf{U}(H')=A(G_{H'})$. The construction of $G_{H'}$ from $G_H$ is given below.

			\textbf{Case a.}  If $|S|>1$, then all the subsets of $X$ excluding $\{u\}$, $\{v\}$, $\{u,v\}$, $\Phi$ and $X$ are the vertices added to $G_H$ in the construction of $G_{H'}$. Let $G_H^{(1)}$ be the graph obtained by adding the pendant edge $\{\{u\}, S\cup \{v\}\}$ at the vertex $\{u\}$ of $G_H$. Let $G_H^{(2)}$ be the graph obtained by adding the pendant edge $\{\{v\},S\cup \{u\}\}$ at the vertex $\{v\}$ of $G_H^{(1)}$. Let $G_H^{(3)}$ be the graph obtained by adding the pendant edge $\{\{u,v\}, S\}$ at the vertex $\{u,v\}$ of $G_H^{(2)}$.
			
			Also, form a copy of $K_2$ in $G_{H'}$ for each pair of proper subsets of $X$ excluding $\{u\}$, $\{v\}$, $\{u,v\}$, $ S\cup \{v\}$,  $S\cup \{u\}$ and $S$. These copies are mutually disjoint. In this way, totally we get $\alpha_1$ number of such copies, where $\alpha_1$ is given in~\eqref{alpha1}.				
			Then $G_{H'}$ is the disjoint union of $G_H^{(3)}$ and these $\alpha_1$ copies of $K_2$. 
			
			\textbf{Case b.}  If $|S|=1$, then there is no non-empty proper subset of $X$ other than $\{u\}$, $\{v\}$, $\{u,v\}$, $ S\cup \{v\}$,  $S\cup \{u\}$ and $S$. So, $G_{H'}$ is the graph $G_H^{(3)}$.
	
		Applying~\cite[Theorem~2.2.1]{cvetko}: ``Let $G$ be a simple graph with vertices $v_1$, $v_2$, \dots, $v_n$. Let $G_j$ denote the graph obtained from $G$ by adding a pendant edge at the vertex $v_j$, then $P_{A(G_j)}(x)=xP_{A(G)}(x)-P_{A(G\backslash\{v_j\})}(x)"$ to the graphs $G_H^{(1)}$, $G_H^{(2)}$ and $G_H^{(3)}$ successively, we have
		\begin{align*}
		P_{A(G_H^{(1)})}(x)&=xP_{A(G_H)}(x)-P_{A(G_H\backslash\{\{u\}\})}(x),\nonumber\\
		P_{A(G_H^{(2)})}(x)&=x^2P_{A(G_H)}(x)-xP_{A(G_H\backslash\{\{u\}\})}(x))-xP_{A(G_H\backslash \{\{v\}\})}(x)+P_{A(G_H\backslash T)}(x),\nonumber\\
		P_{A(G_H^{(3)})}(x)&=x[x(xP_{A(G_H)}(x)-P_{A(G_H\backslash\{\{u\}\})}(x))-xP_{A(G_H\backslash \{\{v\}\})}(x)+P_{A(G_H\backslash T)}(x)]\nonumber\\
		&\quad-x^2P_{A(G_H\backslash\{\{u,v\}\})}(x)+xP_{A(G_H\backslash T')}(x)+xP_{A(G_H\backslash \{\{v\}\})}(x)-P_{A(G_H\backslash T)}(x),
		\end{align*}
		where $T=\{\{u\},\{v\}\}$, and $T'=\{\{u\},\{u,v\}\}$.
		
		Since $A(G_H)=\mathbf{U}(H)$ and $A(G_H\backslash D)=\mathbf{U}(H)(D|D)$ for any $D\subseteq V(G_H)$, the above equations become
		\begin{align}
		P_{A(G_H^{(1)})}(x)&=xP_{\mathbf{U}(H)}(x)-P_{\mathbf{U}(H)(\{u\}|\{u\})}(x),\nonumber\\
		P_{A(G_H^{(2)})}(x)&=x^2P_{\mathbf{U}(H)}(x)-xP_{\mathbf{U}(H)(\{u\}|\{u\})}(x)-xP_{\mathbf{U}(H)(\{v\}|\{v\})}(x)+P_{\mathbf{U}(H)(T|T)},\label{eq}\\
		P_{A(G_H^{(3)})}(x)&=x^3P_{\mathbf{U}(H)}(x)-x^2P_{\mathbf{U}(H)(\{u\}|\{u\})}(x)-(x^2-x)P_{\mathbf{U}(H)(\{v\}|\{v\})}(x)\nonumber\\&  
		\quad+(x-1)P_{\mathbf{U}(H)(T|T)}(x)-x^2P_{\mathbf{U}(H)(\{u,v\}|\{u,v\})}(x)+xP_{\mathbf{U}(H)(T'|T')}(x).\label{eqq}
		\end{align}	
		Notice that 
		\begin{align}\label{EQ}
		P_{A(G_{H'})}(x)=P_{A(G_H^{(3)})}(x) [P_{A(K_2)}(x)]^{\alpha_1}.
		\end{align}
		So, part (i) follows by substituting~\eqref{eqq} and  $P_{A(K_2)}(x)=x^2-1$ in~\eqref{EQ}.
		
Next, we shall prove part (ii). Suppose there is no edge $e\in E(H)$ with $|e|>2$ such that $u,v\in e$. 		
		Then $\{u,v\}\notin I(H)$.
		Notice that $G_{H'}$ is obtained from $G_H$ by adding the pendant edges $\{\{u\}, S\cup \{v\}\}$ and $\{\{v\}, S\cup \{u\}\}$ at $\{u\}$ and $\{v\}$ respectively, followed by adding disjoint union of $\alpha_2$ copies of $K_2$, where $\alpha_2$ is given in~\eqref{alpha2}, i.e., $G_{H'}$ is the disjoint union of $G_H^{(2)}$, and $\alpha_2$ copies of $K_2$. Thus,
		\begin{align}\label{EQQ}
		P_{A(G_{H'})}(x)=P_{A(G_H^{(2)})}(x) [P_{A(K_2)}(x)]^{\alpha_2}.
		\end{align}  So, part (ii) follows by substituting~\eqref{eq} and  $P_{A(K_2)}(x)=x^2-1$ in~\eqref{EQQ}.
\end{proof}


\section{Bounds}\label{sec9}

\subsection{Unified spectral radius and some hypergraph invariants}
We call the largest eigenvalue of the unified matrix of a hypergraph $H$ as the \textit{unified spectral radius of $H$}. In this section, we provide some bounds on the unified spectral radius of $H$.
\begin{thm}
	Let $H$ be a simple hypergraph with $e$-index $k$. Then,	
	\begin{align*}
		\lambda_1(H)^2\leq\frac{k-1}{k}\left(\left(\underset{S\in I(H)}\sum d_H(S)\right)-\partial(H)\right).
	\end{align*}
\end{thm}
\begin{proof}
	Clearly,
	$\underset{i=1}{\overset{k}{\sum}}\lambda_{i}(H)=0$. Therefore, \begin{equation}\label{eqq1}
		\lambda_1(H)\leq \underset{i=2}{\overset{k}{\sum}}|\lambda_{i}(H)|.
	\end{equation}
	By Cauchy-Schwarz inequality, we have 
	\begin{equation}\label{eqq3}
		\left(\underset{i=2}{\overset{k}{\sum}}|\lambda_{i}(H)|\right)^2\leq \underset{i=2}{\overset{k}{\sum}}|\lambda_{i}(H)|^2\cdot \underset{i=2}{\overset{k}{\sum}}1=(k-1)\underset{i=2}{\overset{k}{\sum}}\lambda_{i}(H)^2.
	\end{equation}
	Applying~\eqref{eqq3} and \eqref{eqq1} in \eqref{L eq2}, we get
	\begin{align*}
		\left(\underset{S\in I(H)}\sum d_H(S)\right)-\partial(H)-\lambda_1(H)^2&=	\underset{i=2}{\overset{k}{\sum}}\lambda_{i}(H)^2\\
	&\geq \frac{1}{k-1}		\left(\underset{i=2}{\overset{k}{\sum}}|\lambda_{i}(H)|\right)^2 \\&\geq\frac{1}{k-1}\cdot \lambda_1(H)^2\nonumber
	\end{align*}
	and so,
	$\frac{k-1}{k}\left(\left(\underset{S\in I(H)}\sum d_H(S)\right)-\partial(H) \right)\geq \lambda_1(H)^2.$
\end{proof}	
\begin{thm}
	Let $H$ be a simple hypergraph. Then, $|\lambda(H)|\leq \Delta(H)$ for any eigenvalue $\lambda(H)$ of $\mathbf{U}(H)$.
\end{thm}
\begin{proof}		 
	Let $S\in I(H)$. Notice that for any $v\in S$, we have $d_{G_H}(\{v\})\geq d_{G_H}(S)$. Therefore, the maximum degree of $G_H$ occurs at a vertex corresponding to a singleton element of $I(H)$.
	Since $H$ has no loops, $d_H(v)=d_{G_H}(\{v\})$ and so $\Delta(H)=\underset{v\in V(H)}{\max}d_H(v)= \Delta(G_H)$.
	Hence, the proof follows from~\cite[Proposition 1.1.1]{cvetko}: ``If a simple graph $G$ has the maximum degree $\Delta(G)$, then $|\lambda(G)|\leq\Delta(G)$ for any eigenvalue $\lambda(G)$ of $A(G)$".
\end{proof}
In particular, for a simple hypergraph $H$, $\lambda_1(H)\leq \Delta(H)$.
\begin{thm}\label{bound t}
	Let $H$ be a hypergraph with $e$-index $k$. Then, 
	$\lambda_1(H)\geq~\delta^*(H)-\frac{\partial(H)}{k}.$
\end{thm}
\begin{proof}
	Let $x$ be a column vector in $\mathbb{R}^k$ with the usual Euclidean norm. We consider the extremal representation,
	\begin{align*}
		\lambda_1(H)=\underset{||x||=1}{\max}\bigl\{x^T\mathbf{U}(H)x\bigr\}=\underset{||x||\neq0}{\max}\biggl\{\frac{x^T\mathbf{U}(H)x}{x^Tx}\biggr\}.
	\end{align*} 
	Thus,\begin{align*}
		\lambda_1(H)&\geq \frac{J^T_{k\times 1} \mathbf{U}(H) J_{k\times 1}}{J_{k\times 1}^T J_{k\times 1}}\\
		&=\frac{1}{k}\left[\left(\underset{S\in I(H)}{\sum} d_H(S)\right)-\partial(H)\right]\\
		&\geq \delta^*(H)-\frac{\partial(H)}{k}.
	\end{align*}
	This completes the proof.	
\end{proof}

\begin{thm}
	Let $H$ be a hypergraph with rank $m$ and $e$-index $k$. Let $G$ be a simple spanning subgraph of $H$. Then, $$(3-m)\lambda_1(H)+(m-2)\lambda_k(H)\leq \lambda_1(G).$$
\end{thm}
\begin{proof}
	For each $r=1,2,\dots, m-1$, there exists $S\in I(H)$ of cardinality $r$, since $rank(H)=m$. Let $\mathbf{U}(H)$ be partitioned into $(m-1)^2$ blocks $M_{ij}:=\mathbf{U}(H)[\mathcal{S}_i|\mathcal{S}_j]$ for $i,j=1,2,\dots,m-1$, where $\mathcal{S}_t=\{S\in I(H)~|~|S|=t\}$ for $t=1,2,\dots,m-1$. Then by~\cite[Corollary~1.3.17]{cvetko}, we have
	\begin{equation}\label{ineq}
		\lambda_1(H)+(m-2)\lambda_k(H)\leq \underset{i=1}{\overset{m-1}{\sum}}\lambda_1(M_{ii}).
	\end{equation} 
	By Interlacing Theorem for symmetric matrices~(c.f.~\cite[Theorem~1.3.11]{cvetko}), we have $\lambda_1(\mathbf{U}(H))\geq \lambda_1(M_{ii})$ for all $i=1,2,\dots m-1$. So,~\eqref{ineq} becomes,
	\[\lambda_1(H)+(m-2)\lambda_k(H)\leq \lambda_1(M_{11})+(m-2)\lambda_1(H).\]
	Notice that $M_{11}=\mathbf{U}(G)$ and so $\lambda_1(M_{11})=\lambda_1(G)$, we get the result as desired. 
\end{proof}
\begin{thm}\label{interlace}
	Let $H$ be a hypergraph on $n$ vertices with $e$-index $k_1$ and let $H'$ be an induced subhypergraph of $H$ with $e$-index $k_2$. 
	Then,
	$$\lambda_{k_1-k_2+i}(H)\leq \lambda_i(H')\leq\lambda_i(H)~\text{for}~i=1,2,\dots,k_2.$$
\end{thm}
\begin{proof}
	Since $I(H')\subseteq I(H)$, and any $(S, S')$-th entry in $\mathbf{U}(H')$ is the same as that of $\mathbf{U}(H)$, it follows that the unified matrix of the induced subhypergraph $H'$ is a principal submatrix of $\mathbf{U}(H)$. Thus, the proof directly follows from the Interlacing Theorem for symmetric matrices. 
\end{proof}


\subsection{Chromatic numbers and unified eigenvalues of a hypergraph}
	The following definitions can be found in \cite{voloshin}.
	A \emph{proper coloring} or \emph{weak coloring} of a loopless hypergraph $H$ is an assignment of colors to the vertices of $H$ in such a way that at least two vertices in every edge receive different colors. A proper coloring of a hypergraph $H$ using at most $\lambda$ colors is called a \emph{proper $\lambda$-coloring}. The \emph{chromatic number of $H$}, denoted by $\chi(H)$, is the minimum number of colors required for a proper coloring of $H$.
	
	A \emph{strong coloring} or \emph{strong vertex coloring} of a hypergraph $H$ is a coloring of the vertices of $H$ in such a way that every edge $e\in E(H)$ has all vertices colored differently. The \emph{strong chromatic number} of $H$, denoted by $\chi_s(H)$, is the minimum number of colors required for a strong coloring of $H$.

As a strong vertex coloring of a hypergraph $H$ is always a weak coloring, we have $\chi(H)\leq \chi_s(H)$.
\begin{thm}
	Let $H$ be a loopless hypergraph having no included edges. If for every induced subhypergraph $H'$ of $H$, $\delta^*(H')=d_{H'}(v)$ for some $v\in V(H')$, then $\chi(H)\leq 1+\lambda_1(H).$
\end{thm}
\begin{proof}
	If $\chi(H)=1$, then $H$ is totally disconnected, and so the result is true.
	
	Suppose that $\chi(H)=r\geq2$. Let $H'$ be an induced subhypergraph of $H$ with the  property that $\chi(H')=r$ and $\chi(H'\backslash \{u\})<r$ for any $u\in V(H')$.

	First, we shall show that $\delta(H')\geq r-1$. Assume, for contradiction, that this is not the case. Let $v$ be a vertex of $H'$ such that $d_{H'}(v)<r-1$. Since $\chi(H'\backslash \{v\})<r$ and $H'$ has no included edges, we may extend the coloring of $H'\backslash \{v\}$ to a proper $(r-1)$-coloring of $H'$, a contradiction. So, 
	\begin{equation}\label{del eq}
		\delta(H')\geq r-1.
	\end{equation}
	By assumption,
	\begin{align}\label{eq 3}
		\delta(H')&\leq d_{H'}(v)=\delta^*(H').	
	\end{align}
	Now, by Theorems~\ref{bound t} and~\ref{interlace}, we have
	\begin{align}\label{eq }
		\delta^*(H')\leq\lambda_1(H')\leq \lambda_1(H). 
	\end{align}
	Thus, the result follows from~\eqref{del eq}, \eqref{eq 3} and \eqref{eq }.
\end{proof}

\begin{defn}\normalfont
	A subhypergraph of $H$ which itself is a graph is called a \textit{subgraph of $H$}. A subhypergraph $H'$ of $H$ is called an \textit{induced subgraph of $H$} if all the edges of cardinality two in $E(H)$ which are completely contained in $V(H')$ forms $E(H')$. A subhypergraph $H'$ of $H$ is called a \textit{spanning subgraph of $H$} if $V(H')=V(H)$.
\end{defn}

\begin{thm}
	Let $H$ be a simple hypergraph on $n$ vertices with $e$-index $k$, and having at least one edge of cardinality two. Let $G$ be a spanning induced subgraph of $H$. Then,
	\[\chi_s(H)\geq\chi(H)\geq \chi(G)\geq1+\displaystyle\frac{\lambda_1(G)}{|\lambda_n(G)|}\geq1+\displaystyle\frac{\lambda_1(G)}{|\lambda_{k}(H)|}.\]
\end{thm}
\begin{proof}
	The first inequality is straight forward. Since, a weak coloring of $H$ gives a proper coloring of $G$, we have $\chi(G)\leq \chi(H)$. Notice that $A(G)$ is a principal submatrix of $\mathbf{U}(H)$. Thus, by Interlacing Theorem for symmetric matrices, we have $\lambda_n(G)\geq \lambda_{k}(H)$. Since $H$ has no loops and $G$ has at least one edge, it follows that $\lambda_k(H)$ and $\lambda_n(G)$ are negative and so $|\lambda_n(G)|\leq |\lambda_{k}(H)|$. It is known from~\cite[Theorem 3.10.7]{cvetko} that if $G$ is a simple graph with $n$ vertices and with at least one edge, then $\chi(G)\geq1+\displaystyle\frac{\lambda_1(G)}{|\lambda_n(G)|}$. Hence the remaining two inequalities follows from these facts.
\end{proof}


\subsection{Independence number, complete-clique number and unified eigenvalues of a hypergraph}

\begin{defn}\normalfont
	If $H$ is a hypergraph with $e$-index $k$, then we define the \emph{average unified degree} (or \emph{mean unified degree}) of $H$, denoted by $\bm\bar{d}(H)$, as		\[\bm\bar{d}(H)=\frac{1}{k}\left(\underset{S\in I(H)}{\sum}d(S)\right).\]
\end{defn}
\begin{thm}
	If $H$ is a hypergraph having no included edges, then 
	$\lambda_1(H)\geq\bm\bar{d}(H)$. 
\end{thm}
\begin{proof}	
	Let the $e$-index of $H$ be $k$, and let $x$ be a column vector in $\mathbb{R}^k$ with the usual Euclidean norm. We consider the extremal representation,
	\begin{align*}
		\lambda_1(H)=\underset{||x||=1}{\max}\bigl\{x^T\mathbf{U}(H)x\bigr\}=\underset{||x||\neq0}{\max}\biggl\{\frac{x^T\mathbf{U}(H)x}{x^Tx}\biggr\}.
	\end{align*} 
	Thus,\begin{align*}
		\lambda_1(H)&\geq \frac{J^T_{k\times 1} \mathbf{U}(H) J_{k\times 1}}{J_{k\times 1}^T J_{k\times 1}}.
	\end{align*}
	Since $H$ has no included edges,
	$J^T_{k\times 1} \mathbf{U}(H) J_{k\times 1}=\underset{S\in I(H)}{\sum}d_H(S)$. 
	This completes the proof.
\end{proof}
A subset of vertices of a hypergraph $H$ which contains no edge of $H$ is called an \textit{independent set}.
The \textit{independence number} of $H$, denoted by $\alpha(H)$, is the cardinality of the largest independent set in $H$~(c.f.~\cite[p. 151]{voloshin}).
\begin{thm}
	Let $H$ be a simple hypergraph with $e$-index $k$. Let $\ell^+$ and $\ell^-$ denote the number of positive and negative eigenvalues of $\mathbf{U}(H)$, respectively. Then, $\alpha(H)\leq\min\{k-\ell^+, k-\ell^-\}.$ 
\end{thm}
\begin{proof}
	Notice that if $S$ is an independent set of $H$, then $S'=\{\{v\}~|~v\in S\}$ is an independent set in $G_H$. Hence, $\alpha(H)\leq\alpha(G_H)$. Since $\mathbf{U}(H)=A(G_H)$, the proof follows from~\cite[Theorem 3.10.1]{cvetko}: ``Let $G$ be a simple graph on $n$ vertices and let $\ell^+$ and $\ell^-$ denote the number of positive and negative eigenvalues of $A(G)$, respectively. Then $\alpha(G)\leq\min\{n-\ell^+, n-\ell^-\}".$ 
\end{proof}
\begin{defn}\normalfont
	Let $H$ be a hypergraph. By a \emph{complete-clique} in $H$, we mean a complete subhypergraph of $H$. The \emph{complete-clique number} of $H$, denoted by $\omega(H)$, is the number of vertices in a largest complete-clique of $H$ if $H$ contains a complete-clique; $0$, otherwise.  
\end{defn}
The proof of the following result is analogous to that of Theorem~3.10.3 of~\cite{cvetko}.
\begin{thm}
	Let $H$ be a hypergraph without multiple edges. Let $n^-, n^0, n^+$ denote the number of eigenvalues of $\mathbf{U}(H)$ which are less than, equal to, or greater than $0$, respectively. Then, $\omega(H)\leq \min\{n^-+n^0+1,~ n^0+n^+,~ \lambda_1(H)\}.$
\end{thm}
\begin{proof}
	Let the $e$-index of $H$ be $k$. Suppose $H$ has no complete-clique, then $tr(U(H))=0$ and so $\lambda_1(H)>0$. Thus the result is true in this case. Suppose $H$ has a complete-clique on $p$ vertices. As $H$ has no multiple edges, it follows that $J_{p\times p}$ is a principal submatrix of $\mathbf{U}(H)$.  By Interlacing Theorem for $\mathbf{U}(H)$ and $J_{p\times p}$, we get $\lambda_{k-p+i}(H)\leq \lambda_i(J_{p\times p})\leq \lambda_i(H)$ for $i=1,2,\dots,k$. Since the eigenvalues of $J_{p\times p}$ are $p$ with multiplicity $1$, and $0$ with multiplicity $p-1$, we have
	\begin{eqnarray}
		\lambda_{k-p+1}(H)\leq p\leq \lambda_1(H),\label{clique eqn 1}& \\
		\lambda_{k-p+i}(H)\leq 0\leq \lambda_i(H)&\text{for}~i=2,3,\dots,p.\label{clique eqn 2}
	\end{eqnarray}
	From the L.H.S of~\eqref{clique eqn 2}, we have
	\begin{equation}
		p\leq n^0+n^-+1.\label{clique eqn 3}
	\end{equation}
	From the R.H.S of~\eqref{clique eqn 1} and \eqref{clique eqn 2}, we obtain
	\begin{equation}
		p\leq n^0+n^+.\label{clique eqn 4}
	\end{equation}
	Hence the result follows from~\eqref{clique eqn 3}, \eqref{clique eqn 4} and from the second inequality of~\eqref{clique eqn 1}.
\end{proof}


\section*{Conclusion}

The unified matrix associated with a hypergraph provides a unified approach for linking spectral hypergraph theory with  the spectra of the adjacency matrices of graphs. In this context,  we introduced certain hypergraph structures and invariants, such as exact walk, exact path, exact cycle, unified path, unified cycle, exactly connectedness, exact distance, exact girth, and exact diameter, and related them to the eigenvalues of the unified matrix. Although the relationships between these structures and invariants, and the spectrum of the unified matrix of the hypergraph have been established, further research is needed to explore their properties from a non-spectral hypergraph theoretical viewpoint. Moreover, this approach allows for the extension of various results and properties of graphs, typically expressed using adjacency matrices, to hypergraphs using the unified matrix.


\end{document}